\pdfoutput=1
\RequirePackage{ifpdf}
\ifpdf % We are running pdfTeX in pdf mode
\documentclass[pdftex]{sigma}
\else
\documentclass{sigma}
\fi

\usepackage{tikz, enumerate, braids}
\usepackage{multirow}
\usepackage[all]{xy}
\DeclareMathOperator{\spp}{supp}

\numberwithin{equation}{section}

\newtheorem{Theorem}{Theorem}[section]
\newtheorem{Corollary}[Theorem]{Corollary}
\newtheorem{Lemma}[Theorem]{Lemma}
\newtheorem{Proposition}[Theorem]{Proposition}
 { \theoremstyle{definition}
\newtheorem{Definition}[Theorem]{Definition}
\newtheorem{Example}[Theorem]{Example}
\newtheorem{Remark}[Theorem]{Remark} }

\begin{document}

\allowdisplaybreaks

\newcommand{\arXivNumber}{1804.09603}

\renewcommand{\PaperNumber}{134}

\FirstPageHeading

\ShortArticleName{A Product on Double Cosets of $B_\infty$}

\ArticleName{A Product on Double Cosets of $\boldsymbol{B_\infty}$}

\Author{Pablo GONZALEZ PAGOTTO}

\AuthorNameForHeading{P.~Gonzalez Pagotto}

\Address{Institut Fourier, Universit\'e Grenoble Alpes, Grenoble, France}
\Email{\href{mailto:pablo.gonzalez-pagotto@univ-grenoble-alpes.fr}{pablo.gonzalez-pagotto@univ-grenoble-alpes.fr}}
\URLaddress{\url{https://www-fourier.ujf-grenoble.fr/~gonzpabl/}}

\ArticleDates{Received May 28, 2018, in final form December 14, 2018; Published online December 27, 2018}

\Abstract{For some infinite-dimensional groups $G$ and suitable subgroups $K$ there exists a~monoid structure on the set $K\backslash G/K$ of double cosets of $G$ with respect to $K$. In this paper we show that the group $B_\infty$, of the braids with finitely many crossings on infinitely many strands, admits such a structure.}

\Keywords{Braid group; double cosets; Burau representation}

\Classification{20F36; 20M99; 20C99}

\section{Introduction}

\subsection{Motivation}
 For some infinite-dimensional groups $G$ and suitable subgroups $K$ there exists a monoid structure on the set $K\backslash G /K$ of double cosets of $G$ with respect to $K$. This can be seen, for example, for the group $S_\infty$ of the finitely supported permutations of $\mathbb{N}$, for infinite-dimensional classical Lie groups, for groups of automorphisms of measure spaces and for $\operatorname{Aut}(F_\infty)$, a direct limit of the groups of automorphisms of the free groups $F_n$.

 The study of these structures was pioneered by R.S.~Ismagilov, followed by G.I.~Ol'shanski, who used them in the representation theory of infinite-dimensional classical Lie groups \cite{okounkov, olsh, olsh2, olsh3, olsh5, olsh6}. More recently there is the work of Yu.A.~Neretin for the infinite tri-symmetric group and $\operatorname{Aut}(F_\infty)$ \cite{neretin5, neretin,neretin6, neretin2}.

 In this paper we show that the group $B_\infty$, of the finite braids on infinitely many strands, admits such a structure. Furthermore, we show how this multiplicative structure is related to similar constructions in $\operatorname{Aut}(F_\infty)$ and ${\rm GL}(\infty)$. We also define a one-parameter generalization of the usual monoid structure on the set of double cosets of ${\rm GL}(\infty)$ (see \cite{neretin3, neretin8}) and show that the Burau representation provides a functor between the categories of double cosets of~$B_\infty$ and~${\rm GL}(\infty)$.

\subsection{The infinite braid group and double cosets}
The \emph{Artin braid group} on $n$ strings $B_n$ \cite{artin, dehornoy, geckpf} has the presentation with $n-1$ generators $\sigma_1, \sigma_2, \ldots, \sigma_{n-1}$ and the so-called \emph{braid relations}:
\begin{gather*}
\sigma_i\sigma_j = \sigma_j\sigma_i, \qquad \vert i-j \vert \geq 2,\qquad i,j \in \{1,\ldots,n-1\},
\end{gather*}and \begin{gather*}
\sigma_i\sigma_{i+1}\sigma_i = \sigma_{i+1}\sigma_i\sigma_{i+1}, \qquad 1\leq i \leq n-2.
\end{gather*}
The generators $\sigma_i$ are called \emph{elementary braids}. For each $n$, consider the monomorphism $i_n\colon B_n \to B_{n+1}$ sending the $k$-th elementary braid of $B_n$ to the $k$-th elementary braid of $B_{n+1}$. Geometrically this operation corresponds to adding a new string to the right of the braid, without creating any new crossings, as in the picture below:
\begin{figure}[h!]\centering
\begin{tikzpicture}
\draw (1,-.5)--(1,1.5);
\draw (2,-.5)--(2,1.5);
\draw (2.5,-.5)--(2.5,1.5);
\draw (3,-.5)--(3,1.5);
\draw (3.5,-.5)--(3.5,1.5);

\draw[fill=black!10] (.9,0) rectangle (3.6,1);

\draw (1.5,1.25) node{\ldots};
\draw (1.5,-.25) node{\ldots};
\draw (2.25,.5) node{braid};
\draw[->] (4,.5)--(4.5,.5);

\begin{scope}[shift={(4,0)}]
\draw (1,-.5)--(1,1.5);
\draw (2,-.5)--(2,1.5);
\draw (2.5,-.5)--(2.5,1.5);
\draw (3,-.5)--(3,1.5);
\draw (3.5,-.5)--(3.5,1.5);

\draw (4,-.5)--(4,1.5);

\draw[fill=black!10] (.9,0) rectangle (3.6,1);

\draw (1.5,1.25) node{\ldots};
\draw (1.5,-.25) node{\ldots};
\draw (2.25,.5) node{braid};
\end{scope}
\end{tikzpicture}
\caption{The monomorphism $i_n$.}
\end{figure}
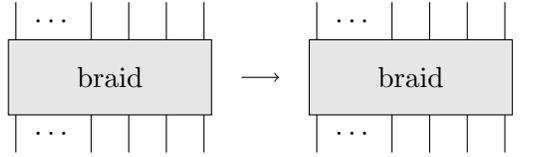

The direct limit of this sequence of groups, with respect to the homomorphisms $i_n$, is the \emph{infinite braid group} \begin{gather*}
B_\infty = \lim_{\longrightarrow} B_n,
\end{gather*}
consisting of braids with countably many strings and finitely many crossings. This group has the presentation \begin{gather*}
B_\infty = \left\langle \begin{array}{@{}l|c@{}}
\multirow{2}{*}{$\sigma_i, i\in \mathbb{N}$}&\sigma_i\sigma_j = \sigma_j\sigma_i, \vert i-j \vert \geq 2\\
& \sigma_i\sigma_{i+1}\sigma_i = \sigma_{i+1}\sigma_i\sigma_{i+1} \end{array}\right\rangle.
\end{gather*}
For each non-negative integer $\alpha$, let $B_\infty[\alpha]$ be the subgroup of $B_\infty$ given by
\begin{gather*}
B_\infty[\alpha] = \langle\sigma_j\,|\,j>\alpha\rangle.
\end{gather*}

\begin{Definition}Let $G$ be a group, $g\in G$ and $K$ and $L$ be subgroups of $G$. The double coset on $G$ containing $g$ with respect to the pair $(K,L)$ is the set $KgL$. Denote by $K\backslash G/L$ the set of double cosets on $G$ with respect to the pair $(K,L)$.
\end{Definition}

\subsection[The Burau representation of $B_\infty$]{The Burau representation of $\boldsymbol{B_\infty}$}\label{sec:burau}
The Burau representation \cite{birman, birble} is the homomorphism $\eta_n\colon B_n\to {\rm GL}\big(n,\mathbb{Z}\big[t,t^{-1}\big]\big)$ given by \begin{gather*}
\eta_n(\sigma_i)= \begin{pmatrix}
{1}_{i-1}&&\\
&\begin{pmatrix}
1-t&t\\
1&0
\end{pmatrix}&\\
&&{1}_{n-i-1}
\end{pmatrix}.
\end{gather*}
Denote ${\rm GL}\big(n,\mathbb{Z}\big[t,t^{-1}\big]\big)$ by ${\rm GL}(n)$ and consider the homomorphisms $j_n\colon {\rm GL}(n) \to {\rm GL}(n+1)$ given by \begin{gather*}
j_n(T) = \begin{pmatrix}
T&0\\
0&1
\end{pmatrix}.
\end{gather*} The group ${\rm GL}(\infty)$ is the direct limit of ${\rm GL}(n)$ with respect to the homomorphisms $j_n$ and consists of infinite matrices that differ from the identity matrix only in finitely many entries. Due to the commutativity of the diagram \begin{gather*}
\xymatrix{B_n\ar[r]^-{\eta_n}\ar[d]_{i_n}&{\rm GL}(n)\ar[d]^{j_n}\\
B_{n+1}\ar[r]_-{\eta_{n+1}}&{\rm GL}(n+1)}
\end{gather*} we can construct a representation $\eta\colon B_\infty \to {\rm GL}(\infty)$ of $B_\infty$ by taking the limit of the representations $\eta_n$. More precisely, $\eta$ is given by the following formulas: \begin{gather*}
\eta(\sigma_i)= \begin{pmatrix}
{1}_{i-1}&&\\
&\begin{pmatrix}
1-t&t\\
1&0
\end{pmatrix}&\\
&&{1}_{\infty}
\end{pmatrix}.
\end{gather*}
With this representation in mind, we will define an operation on double cosets of ${\rm GL}(\infty)$ such that the Burau representation will be functorial between the categories of double cosets.

\subsection{Main results}

Consider the double cosets on $B_\infty$ with respect to the subgroups $B_\infty[\alpha]$. Given double cosets $\mathfrak{p}\in B_\infty[\alpha]\backslash B_\infty/B_\infty[\beta]$ and $\mathfrak{q}\in B_\infty[\beta]\backslash B_\infty/B_\infty[\gamma]$, we are going to define an element $\mathfrak{p\circ q}\in B_\infty[\alpha]\backslash B_\infty/B_\infty[\gamma]$. To this purpose, we first introduce the following:

\begin{Definition}
For integers $\beta\geq 0$ and $n> 0$, denote by $\tau_i^{(n)}$ the braid \begin{gather*}
\tau_i^{(n)} = \sigma_{n+\beta+i}\sigma_{n+\beta+i-1}\cdots\sigma_{\beta+i+1}.
\end{gather*} Further we define the element $\theta_n[\beta]\in B_\infty[\beta]$ as \begin{gather*}
\theta_n[\beta] = \tau_0^{(n)}\tau_1^{(n)}\cdots\tau_{n-1}^{(n)}.
\end{gather*}
\end{Definition}

\begin{figure}[h!]\centering
\begin{tikzpicture}[scale=0.7]
\draw[rounded corners=3pt, fill=blue!40, draw=none] (9.95,-2.2)--(11.05,-2.2)--(11.05,-3.3)--(7.05,-7.3)--(5.95,-7.3)--(5.95,-6.2)--cycle;
\braid
s_{8}
s_{7}-s_{9}
s_{6}-s_{8}-s_{10}
s_{5}-s_{7}-s_{9}-s_{11}
s_{4}-s_{6}-s_{8}-s_{10}-s_{12}
s_{5}-s_{7}-s_{9}-s_{11}
s_{6}-s_{8}-s_{10}
s_{7}-s_{9}
s_{8}
;
\end{tikzpicture}
\caption{The element $\theta_5[3]$. The highlighted region corresponds to $\tau_2^{(5)}$.}
\end{figure}
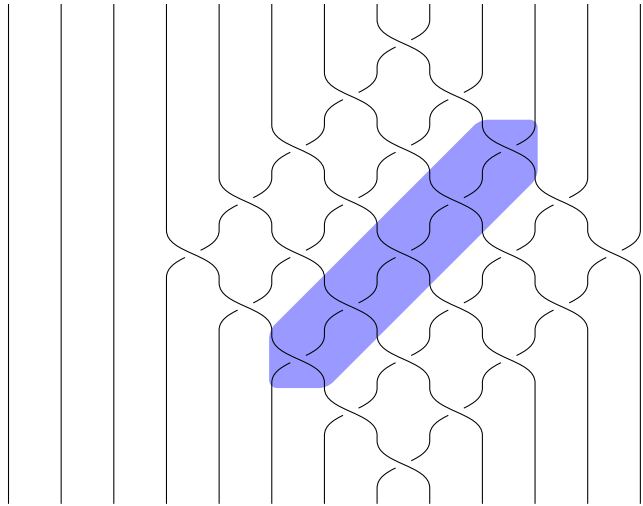

Finally, the definition of the product of the double cosets is as follows:

\begin{Definition}Let $\mathfrak{p}\in B_\infty[\alpha]\backslash B_\infty /B_\infty[\beta]$ and $\mathfrak{q}\in B_\infty[\beta]\backslash B_\infty/B_\infty[\gamma]$ be double cosets. Consider $p\in \mathfrak{p}$ and $q\in\mathfrak{q}$ representatives of these double cosets. Then we define their product as \begin{gather*}
\mathfrak{p\circ q} = B_\infty[\alpha] p\theta_n[\beta]q B_\infty[\gamma],
\end{gather*} for sufficiently large $n$.\label{def:prodnavy}
\end{Definition}

\begin{Remark} The introduction of the element $\theta_k[\beta]$ is essential for our construction of a~product on the set of double cosets of $B_\infty$. In fact, for $\mathfrak{p}\in B_\infty[\alpha]\backslash B_\infty/B_\infty[\beta]$ and $\mathfrak{q}\in B_\infty[\beta]\backslash B_\infty/B_\infty[\gamma]$, let $p\in \mathfrak{p}$ and $q\in \mathfrak{q}$ be representatives of these double cosets. The ``naive'' product $ B_\infty[\alpha] pq B_\infty[\gamma]$ does not always coincide for all choices of $p$ and $q$. For instance $\sigma_2$ and $\sigma_3\sigma_2$ are representatives of the same double coset in $B_\infty[2]\backslash B_\infty/B_\infty[2]$. But $\sigma_2^2$ and $\sigma_3\sigma_2\sigma_3\sigma_2$ represent distinct cosets. In order to see this, we consider the permutation associated to each braid. For the braid $\sigma_2^2$ it is the identity and for the braid $\sigma_3\sigma_2\sigma_3\sigma_2$ it is $(432)$. Since no braid in~$B_\infty[2]$ permutes the point~$2$, we see that these are in fact distinct double cosets.

However, if we introduce an intermediary braid $\theta_k[\beta]$ that ``forces apart'' the braids $p$ and $q$, the double coset $B_\infty[\alpha] p\theta_k[\beta]q B_\infty[\gamma]$ becomes independent of $k$ for $k$ large enough and its limit does not depend on the choice of the representatives for $\mathfrak{p}$ and $\mathfrak{q}$.
\end{Remark}

\begin{Theorem}\label{theorem:opbr}The operation defined above does not depend on the choice of the representatives of the double cosets for~$n$ large enough. Moreover, it is associative.
\end{Theorem}

As a consequence we have that $(B_\infty[\alpha]\backslash B_\infty/B_\infty[\alpha],\circ)$ is a monoid, for each non-negative integer $\alpha$.

\begin{Remark}We will show that there exists some $n_0(\alpha,\gamma,p,q)$ such that, for all $n\geq n_0$, $B_\infty[\alpha] p\theta_n[\beta]q B_\infty[\gamma] = B_\infty[\alpha] p\theta_{n_0}[\beta]q B_\infty[\gamma]$ . More precisely, $n_0 = \max\{\spp p,\spp q, \alpha,\gamma\}+1$, where $\spp$ is the support of a braid, defined in Definition~\ref{defi:sppt}.\end{Remark}

For a group $G$ and a subgroup $H\subset G$, we say that $g,g'\in G$ are conjugate with respect to~$H$ if there exists $h\in H$ such that $g' = hgh^{-1}$. Denote by $G//H$ the set of conjugacy classes with respect to the subgroup $H$. There is a natural one-to-one correspondence between the sets $H\backslash (G^n\times H)/H$ and $G^n//H$ (here $H$ is the image of the subgroup $H\subset G$ by the appropriate diagonal map). In fact, it is easy to see that the function from $G^n\times H$ to $G^n$ given by \begin{gather*}
 (g_1,g_2,\ldots,g_n, h)\mapsto \big(g_1h^{-1},g_2h^{-1}, \ldots, g_nh^{-1}\big), \end{gather*}
induces a bijection between the sets $H\backslash( G^n\times H)/H$ and $G^n//H$.

Using the correspondence above, we can define a monoid structure on the set $B_\infty//B_\infty[\alpha]$. In fact, we have an one-to-one correspondence between the sets $B_\infty //B_\infty[\alpha]$ and $B_\infty[\alpha]\backslash(B_\infty \times B_\infty[\alpha])/B_\infty[\alpha]$, the later being a submonoid of $B_\infty[\alpha]\backslash(B_\infty \times B_\infty)/B_\infty[\alpha]$.

Furthermore, as a consequence of the existence of a solution for the conjugacy problem for the braid groups and the fact that the injections $i_n$ do not merge conjugacy classes (see~\cite{meneses}), we have
\begin{Proposition}\label{prop:mene} The conjugacy problem for $B_\infty$ has a solution.
\end{Proposition}

Notice that combining the observations above with Proposition~\ref{prop:mene}, it is possible to devise an algorithm to determine when two elements of $B_\infty\times B_\infty$ belong to the same class in $B_\infty[0]\backslash B_\infty\times B_\infty/B_\infty[0]$.

Now, let $v=\big(1,t,t^2,\ldots\big)$ and $u = (1,1,1,1,\ldots)$ and denote by $x^{\rm T}$ the transpose of the vector~$x$. Consider the subgroup of ${\rm GL}(\infty)$ given by \begin{gather*}
G[n] = \left\{\begin{pmatrix} 1_n&\\ &X \end{pmatrix};\,X\in {\rm GL}(\infty),\, v^{\rm T}X = v^{\rm T} ,\, Xu=u\right\}.
\end{gather*}It is easy to see that the image of $B_\infty[n]$ by the Burau representation is contained in~$G[n]$.

\begin{Definition}Consider the matrix \begin{gather*}
\Theta_j[k]=\begin{pmatrix}
\mathbf{1}_k&0&0&0\\
0&{V}_j&t^j{1}_j&0\\
0&{1}_j&0&0\\
0&0&0&\mathbf{1}_\infty
\end{pmatrix},
\end{gather*} where \begin{gather*} {V}_j = (1-t)\begin{pmatrix}
1&t&\cdots&t^{j-1}\\
\vdots&\vdots&\ddots&\vdots\\
1&t&\cdots&t^{j-1}
\end{pmatrix}.\end{gather*}
Let $p\in \mathfrak{p}$ and $q\in \mathfrak{q}$ be representatives of the double cosets $\mathfrak{p}\in G[n]\backslash {\rm GL}(\infty)\slash G[k]$ and $\mathfrak{q}\in G[k]\backslash {\rm GL}(\infty)\slash G[m]$. Then we define their product as \begin{gather*}
\mathfrak{p}\star_t\mathfrak{q} = {G}[n] p\Theta_j[k]q {G}[m],
\end{gather*}for sufficiently large $j$.
\end{Definition}

\begin{Theorem}\label{theorem:opma}The operation defined above does not depend on the choice of representatives of double cosets for $j$ large enough. Moreover, it is associative.
\end{Theorem}
\begin{Remark}In particular, there exists an integer $j_0(n,m,p,q)$ such that $G[n] p\Theta_j[k]q G[m] = G[n] p\Theta_{j_0}[k]q G[m]$ for all $j\geq j_0$. We can make $j_0$ more precise. In fact let $N\in\mathbb{N}$ be such that $p$ and $q$ can be written as diagonal block matrices $\left(\begin{smallmatrix}
A&0\\ 0&1_\infty
\end{smallmatrix}\right)$, where $A$ is a square matrix of dimension $k+N$. Then $j_0 = \max\{m,n,k+N\}$.\end{Remark}

\begin{Remark}The operation $\star_t$ generalizes the usual multiplication defined on the double cosets of ${\rm GL}(\infty)$ in the sense that setting the parameter $t=1$ we recover the usual multiplication.
\end{Remark}

Let $G$ be a group and $K[\ast] = \{K[s]; s\in \mathbb{N}\}$ a family of subgroups of $G$. We say that there is a well-defined operation on the double cosets of $G$ with relation to the family $K[\ast]$ when there exists a family of morphisms \begin{gather*}\mu = \{\mu_{rst}\colon K[r]\backslash G/K[s]\times K[s]\backslash G/K[t]\to K[r]\backslash G/K[t]; r,s,t\in\mathbb{N}\}\end{gather*} satisfying \begin{gather*}\mu_{rtu}(\mu_{rst}\times 1_{K[t]\backslash G/K[u]}) = \mu_{rsu}(1_{K[r]\backslash G/K[s]}\times \mu_{stu})\end{gather*} for all $r,s,t,u \in\mathbb{N}$ and, if $\mathbf{e}\in K[r]\backslash G/K[r]$ denotes the class of the unit element of~$G$, then for all $\alpha\in K[t]\backslash G/K[r]$ and all $\beta \in K[r]\backslash G/K[t]$ \begin{gather*}\mu_{trr}(\alpha, \mathbf{e}) = \alpha\qquad\text{and}\qquad\mu_{rrt}(\mathbf{e},\beta) = \beta.\end{gather*}
In this case, consider the category $\mathcal{K}(G,K)$ of double cosets, where the objects are nonnegative integers and the morphisms are given by $\mathrm{Hom}(r,s) = K[s]\backslash G\slash K[r]$. Then,
\begin{Proposition}The Burau representation $\eta\colon B_\infty\to {\rm GL}(\infty)$ induces a functor between the categories $\mathcal{K}(B_\infty,B_\infty[*])$ and $\mathcal{K}({\rm GL}(\infty),G[*])$.
\end{Proposition}

When $G$ is the bisymmetric group (the group that consists of pairs $(g,h)$ of permutations of~$\mathbb{N}$ such that $gh^{-1}$ is a finite permutation) and $K$ is its diagonal subgroup, we get a category called the train category of the pair~$(G,K)$. This category admits a transparent combinatorial description and encodes information about the representations of the bisymmetric group (see~\mbox{\cite{neretin4, olsh4}}).

\section{Proofs of main results}
\subsection{Proof of Theorem \ref{theorem:opbr}}
Before proceeding, we introduce the notion of \textit{support}, which will be needed later.

\begin{Definition}\label{defi:sppt} Let $p$ be a braid in $B_\infty$. The \textit{support} of $p$ is \begin{gather*}
\spp p = \min\{j\in\mathbb{N};\, p\in\langle \sigma_1,\ldots,\sigma_j\rangle\}.
\end{gather*}\end{Definition}
Notice that the decomposition of $p$ into the product of elementary braids does not contain any element of $B_\infty[\spp p]$, hence $p$ commutes with every element of $B_\infty[1+\spp p]$. Also, we can identify $p$ with an element of $B_{1+\spp p}$. We define $\spp 1 = 0$.

Consider double cosets\begin{gather*}
\mathfrak{p}\in B_\infty[\alpha]\backslash B_\infty\slash B_\infty[\beta] \qquad\text{and}\qquad \mathfrak{q}\in B_\infty[\beta]\backslash B_\infty\slash B_\infty[\gamma],
\end{gather*} and let $p\in\mathfrak{p}$ and $q\in\mathfrak{q}$ be their respective representatives. Setting $\mathfrak{r}_j = B_\infty[\alpha] p\theta_j[\beta]q B_\infty[\gamma],$ we have a sequence of double cosets in $B_\infty[\alpha]\backslash B_\infty\slash B_\infty[\gamma]$.

\begin{Proposition}\label{prop:suflarg} The sequence $(\mathfrak{r}_j)_{j\geq 1}$ defined above is eventually constant.\end{Proposition}
\begin{proof}We are going to give a proof in several steps:

\textit{Step 1.} Given $m>0$ we have $\tau_i^{(m+1)} = \sigma_{m+\beta +1+i}\tau_i^{(m)}$ for all $0\leq i\leq m-1$. In fact, we have the equality \begin{gather*}
\tau_i^{(m+1)} = \sigma_{m+1+\beta+i}\sigma_{m+\beta+i}\sigma_{m+\beta+i-1}\cdots\sigma_{\beta+i+1} = \sigma_{m+1+\beta+i}\tau_i^{(m)}.
\end{gather*}

\textit{Step 2.} For all $j\leq i$ we have $\sigma_{m+\beta+i+2}\tau^{(m)}_j = \tau_j^{(m)} \sigma_{m+\beta+i+2}$. Indeed,
since $\spp \tau_j^{(m)} = m+\beta+j$ and $\sigma_{m+\beta+i+2}\in B_\infty[1+m+\beta+j]$, we find that $\sigma_{m+\beta+i+2}$ commutes with $\tau_j^{(m)}$.

\textit{Step 3.} Define $u=(\sigma_{m+\beta +1}\sigma_{m+\beta+2}\cdots\sigma_{2m+\beta})^{-1}$ and $\ell^{-1}\! = \tau_m^{(m+1)}$. Then $\theta_m[\beta] = u\theta_{m+1}[\beta]\ell$. In fact, we have
\begin{gather*}
u\theta_{m+1}[\beta]\ell = u\big(\tau_0^{(m+1)}\cdots\tau_m^{(m+1)}\big)\ell = u\tau_0^{(m+1)}\cdots\tau_{m-1}^{(m+1)} \\
\hphantom{u\theta_{m+1}[\beta]\ell}{} = u\big(\sigma_{m+\beta +1}\tau_0^{(m)}\big)\big(\sigma_{m+\beta +2}\tau_1^{(m)}\big)\cdots\big(\sigma_{2m+\beta}\tau_{m-1}^{(m)}\big) \\
\hphantom{u\theta_{m+1}[\beta]\ell}{} = \sigma_{2m+\beta}^{-1}\sigma_{2m+\beta-1}^{-1}\cdots\sigma_{m+\beta+2}^{-1}\tau_0^{(m)}\big(\sigma_{m+\beta +2}\tau_1^{(m)}\big)\cdots\big(\sigma_{2m+\beta}\tau_{m-1}^{(m)}\big) \\
\hphantom{u\theta_{m+1}[\beta]\ell}{} = \sigma_{2m+\beta}^{-1}\sigma_{2m+\beta-1}^{-1}\cdots\sigma_{m+\beta+2}^{-1}\sigma_{m+\beta +2}\tau_0^{(m)}\tau_1^{(m)}\cdots\big(\sigma_{2m+\beta}\tau_{m-1}^{(m)}\big) \\
\hphantom{u\theta_{m+1}[\beta]\ell}{} = \sigma_{2m+\beta}^{-1}\sigma_{2m+\beta-1}^{-1}\cdots\sigma_{m+\beta+3}^{-1}\tau_0^{(m)}\tau_1^{(m)}\big(\sigma_{m+\beta +3}\tau_2^{(m)}\big)\cdots\big(\sigma_{2m+\beta}\tau_{m-1}^{(m)}\big) \\
\hphantom{u\theta_{m+1}[\beta]\ell=}{} \cdots\cdots\cdots\cdots\cdots\cdots\cdots\cdots\cdots\cdots\cdots\cdots\cdots\cdots\cdots\cdots\cdots\cdots\cdots\\
\hphantom{u\theta_{m+1}[\beta]\ell}{} = \tau_0^{(m)}\tau_1^{(m)}\tau_2^{(m)}\cdots\tau_{m-1}^{(m)}= \theta_m[\beta].
\end{gather*}

\textit{Step 4.} Let $M = \max\{\spp p,\spp q, \alpha,\gamma\}+1$. We show that for all $m\geq M$, we have $\mathfrak{r}_m = \mathfrak{r}_{m+1}$ and hence that $\mathfrak{r}_m = \mathfrak{r}_M$ for all $m\geq M$. Let $u$ and $\ell$ be like in step 3. Since $u,\ell \in B_\infty[{m+\beta}]$, it follows that $u\in B_\infty[\alpha]$, $\ell\in B_\infty[\gamma]$ and they commute with $p$ and $q$. Therefore \begin{gather*}u(p\theta_{m+1}[\beta]q)\ell = p(u\theta_{m+1}[\beta]\ell)q = p\theta_m[\beta]q.\end{gather*} Thus $\mathfrak{r}_m = \mathfrak{r}_{m+1}$.
\end{proof}

The following technical lemma will be used in the proof of Lemma \ref{lem:comb}, which in turn is used in Proposition \ref{prop:rep} and more extensively in Theorem \ref{th:asso}.

\begin{Lemma}\label{lem:sym} Let $\big\{\big(\upsilon_j^i\big)_{j=1}^\ell\big\}_{i=1}^g$ be a family of sequences of positive integers such that $\upsilon_{j+k}^i<\upsilon_{j}^{i+n}$ whenever $k+n>0$ with $k,n\in \mathbb{N}$; in other words, the sequences $\big(\upsilon_j^i\big)_{j\geq 1}$ are decreasing and the sequences $\big(\upsilon_j^i\big)_{i\geq 1}$ are increasing. If $\mu_j = \prod\limits_{k=1}^\ell \sigma_{\upsilon_k^j}$ and $\lambda_i = \prod\limits_{k=1}^g \sigma_{\upsilon_i^k}$, then $\mu_1\cdots\mu_g=P=\lambda_1\cdots\lambda_\ell$.
\end{Lemma}

\begin{proof}We prove the lemma by induction on the pair $(g,\ell)$. The statement is trivial for $g=\ell=1$. Assume it is true for $(g,\ell)$, we prove it is true for $(g+1,\ell)$ and $(g,\ell+1)$.

For $(g+1,\ell)$, notice that \begin{gather*}
\prod_{s=1}^{g+1}\prod_{r=1}^{\ell}\sigma_{\upsilon_r^s}=\left(\prod_{s=1}^{g}\prod_{r=1}^{\ell}\sigma_{\upsilon_r^s}\right)\left(\prod_{r=1}^{\ell}\sigma_{\upsilon_r^{g+1}}\right) =\left(\prod_{r=1}^{\ell}\prod_{s=1}^{g}\sigma_{\upsilon_r^{s}}\right)\left(\prod_{r=1}^{\ell}\sigma_{\upsilon_r^{g+1}}\right).
\end{gather*}If $x_r=\prod\limits_{s=1}^g \sigma_{\upsilon_r^{s}}$ we have that $x_r\sigma_{\upsilon_t^{g+1}}=\sigma_{\upsilon_t^{g+1}}x_r$ for $r>t$, this follows from the inequalities $\upsilon_r^{s}<\upsilon_r^{g+1}<\upsilon_t^{g+1}$ for $s<g+1$. Therefore \begin{gather*}
\left(\prod_{r=1}^{\ell}x_r\right)\left(\prod_{r=1}^{\ell}\sigma_{\upsilon_r^{g+1}}\right)=
x_1\cdots x_\ell\sigma_{\upsilon_1^{g+1}}\cdots\sigma_{\upsilon_\ell^{g+1}}=
x_1\sigma_{\upsilon_1^{g+1}}x_2\sigma_{\upsilon_2^{g+1}}\cdots x_\ell\sigma_{\upsilon_\ell^{g+1}}\\
\hphantom{\left(\prod_{r=1}^{\ell}x_r\right)\left(\prod_{r=1}^{\ell}\sigma_{\upsilon_r^{g+1}}\right)}{} =
\prod_{r=1}^\ell x_r\sigma_{\upsilon_r^{g+1}}= \prod_{r=1}^{\ell}\prod_{s=1}^{g+1}\sigma_{\upsilon_r^s}.
\end{gather*}

The proof for the case $(g,\ell+1)$ is analogous.
\end{proof}

\begin{Example}Consider the sequences given by $\upsilon_j^i = 3+i-j$, $1\leq i,j \leq 3$. Let $\mu_i$, $\lambda_i$, $i=1,2,3$ be as in Lemma~\ref{lem:sym}. By the same lemma we have that $\mu_1\mu_2\mu_3 = \lambda_1\lambda_2\lambda_3$. These products are depicted in Fig.~\ref{fig:seqprod}(a) and~(c). Drawing these braids in a more compact form (Fig.~\ref{fig:seqprod}(b)) the equivalence between the products becomes evident.
 \begin{figure}[h!]\centering
\noindent\mbox{%
\begin{minipage}{4.5cm}
\resizebox{4.5cm}{4.5cm}{
\begin{tikzpicture}%[yscale=.2, xscale=.5]
% \draw[dotted] (0,0) -- (20.2,0);
% \draw (0,-.250) -- (20.2,-.250);
% \draw (0.5, -1.25) node{$a$}; 
% \draw (0,-2.25) -- (20.2,-2.25);
% \draw (0.5, -5.75) node{$\theta_4[1]$}; 
% \draw (0,-9.25) -- (20.2,-9.25);
% \draw (0.5, -10.25) node{$b$}; 
% \draw (0,-11.25) -- (20.2,-11.25);
% \draw (0.5, -19.75) node{$\theta_9[2]$}; 
% \draw (0,-28.25) -- (20.2,-28.25);
% \draw (0.5, -30.25) node{$c$}; 
% \draw (0,-32.25) -- (20.2,-32.25);
\braid
s_4 s_6 s_7

s_3 s_5 s_6

s_1 s_2 s_3
;

\end{tikzpicture}}%
\begin{center}
(a) $\mu_1\mu_2\mu_3$
\end{center}
\end{minipage}}

\hspace{.7cm}
\noindent\mbox{%
\begin{minipage}{4.5cm}
\resizebox{4.5cm}{4.5cm}{
\begin{tikzpicture}%[yscale=.2, xscale=.5]
% \draw[dotted] (0,0) -- (20.2,0);
% \draw (0,-.250) -- (20.2,-.250);
% \draw (0.5, -1.25) node{$a$}; 
% \draw (0,-2.25) -- (20.2,-2.25);
% \draw (0.5, -5.75) node{$\theta_4[1]$}; 
% \draw (0,-9.25) -- (20.2,-9.25);
% \draw (0.5, -10.25) node{$b$}; 
% \draw (0,-11.25) -- (20.2,-11.25);
% \draw (0.5, -19.75) node{$\theta_9[2]$}; 
% \draw (0,-28.25) -- (20.2,-28.25);
% \draw (0.5, -30.25) node{$c$}; 
% \draw (0,-32.25) -- (20.2,-32.25);
\braid
s_4-s_6

s_1-s_3-s_5-s_7

s_2-s_6

s_3

% s_2
% s_1-s_3
% s_1-s_3
% s_2

;

\end{tikzpicture}}%
\begin{center}
(b) $P$
\end{center}
\end{minipage}}

\hspace{.7cm}%
\noindent\mbox{%
\begin{minipage}{4.5cm}
\resizebox{4.5cm}{4.5cm}{
\begin{tikzpicture}%[yscale=.2, xscale=.5]
% \draw[dotted] (0,0) -- (20.2,0);
% \draw (0,-.250) -- (20.2,-.250);
% \draw (0.5, -1.25) node{$a$}; 
% \draw (0,-2.25) -- (20.2,-2.25);
% \draw (0.5, -5.75) node{$\theta_4[1]$}; 
% \draw (0,-9.25) -- (20.2,-9.25);
% \draw (0.5, -10.25) node{$b$}; 
% \draw (0,-11.25) -- (20.2,-11.25);
% \draw (0.5, -19.75) node{$\theta_9[2]$}; 
% \draw (0,-28.25) -- (20.2,-28.25);
% \draw (0.5, -30.25) node{$c$}; 
% \draw (0,-32.25) -- (20.2,-32.25);
\braid
s_4 s_3 s_1

s_6 s_5 s_2

s_7 s_6 s_3
;

\end{tikzpicture}}%
\begin{center}
(c) $\lambda_1\lambda_2\lambda_3$
\end{center}
\end{minipage}}

\caption{The equality $\mu_1\mu_2\mu_3 = P = \lambda_1\lambda_2\lambda_3$.}\label{fig:seqprod}
\end{figure}
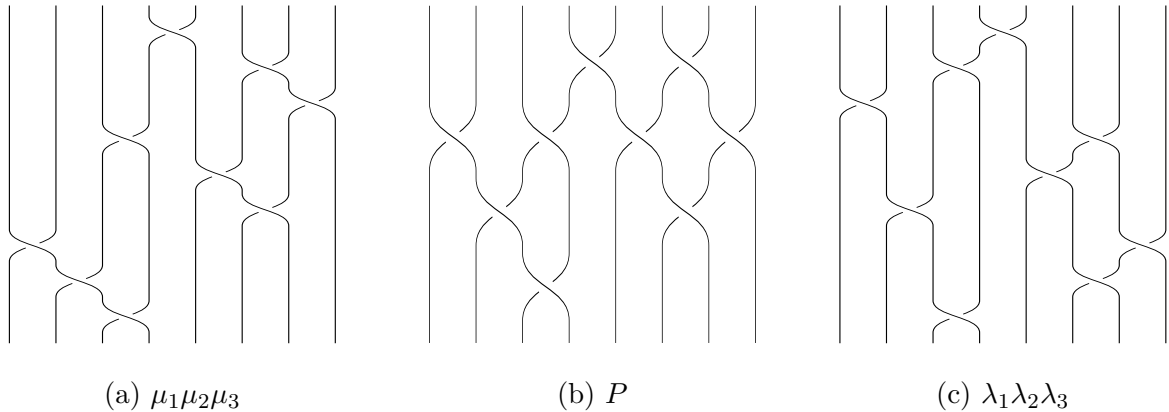
\end{Example}

It will be useful to write the product $P$ from Lemma~\ref{lem:sym} as a matrix, where the indices increase from right to left and from top to bottom.
\begin{gather*}P = \begin{bmatrix}
 \upsilon_1^1&\rightarrow&\upsilon_\ell^1\\
 \downarrow& &\downarrow\\
 \upsilon_1^g&\rightarrow&\upsilon_\ell^g
 \end{bmatrix}.
\end{gather*}

In this way, $\lambda_1\cdots\lambda_\ell$ is the column-wise product and $\mu_1\cdots\mu_g$ is the row-wise product.

Consider, for each positive integer $m$, the homomorphism $C_m\colon B_\infty\to B_\infty$ given by $C_m(\sigma_j)=\sigma_{m+j}$. Then we have the following lemma.

\begin{Lemma}\label{lem:comb} Let $\beta$ and $j$ be nonnegative integers with $j>1$. If $d\in\langle \sigma_{\beta+1},\ldots,\sigma_{\beta+j-1}\rangle$, then:
\begin{enumerate}[$(i)$]\itemsep=0pt
\item $d\theta_j[\beta]=\theta_j[\beta]C_j(d)$,
\item $\theta_j[\beta]d=C_j(d)\theta_j[\beta]$.
\end{enumerate}\end{Lemma}

\begin{proof}Since $C_j$ is a homomorphism, it is enough to prove both statements of the proposition for the case where $d = \sigma_k$, for some $\beta+1\leq k\leq \beta+j-1$.

(i) Recall that $\theta_j[\beta] = \tau_0^{(j)}\cdots\tau_{j-1}^{(j)}$. We claim that the following holds: \begin{gather*}
\sigma_{k+i}\tau_i^{(j)} = \tau_i^{(j)}\sigma_{k+i+1},\qquad 0\leq i\leq j-1.
\end{gather*}
Indeed, since $\sigma_{k+i}$ is a letter of $\tau_i^{(j)},$ but it is different from $\sigma_{j+\beta+i}$, we have
\begin{gather*}
\sigma_{k+i}\tau_i^{(j)} = \sigma_{k+i}(\sigma_{j+\beta+i}\cdots\sigma_{\beta+1+i})=\sigma_{j+\beta+i}\cdots\sigma_{k+i+2}\sigma_{k+i}\sigma_{k+i+1}\sigma_{k+i}\sigma_{k+i-1}\cdots\sigma_{\beta+1+i}\\
\hphantom{\sigma_{k+i}\tau_i^{(j)}}{} =\sigma_{j+\beta+i}\cdots\sigma_{k+i+2}\sigma_{k+i+1}\sigma_{k+i}\sigma_{k+i+1}\sigma_{k+i-1}\cdots\sigma_{\beta+1+i} \\
\hphantom{\sigma_{k+i}\tau_i^{(j)}}{} =\sigma_{j+\beta+i}\cdots\sigma_{\beta+1+i}\sigma_{k+i+1} = \tau_i^{(j)}\sigma_{k+i+1}.
\end{gather*}
Therefore \begin{gather*}
\sigma_k\theta_j[\beta] = \sigma_k\tau_0^{(j)}\cdots\tau_{j-1}^{(j)} = \tau_0^{(j)}\sigma_{k+1}\tau_1^{(j)}\cdots\tau_{j-1}^{(j)} = \cdots = \tau_0^{(j)}\cdots\sigma_{k+j-1}\tau_{j-1}^{(j)} \\
\hphantom{\sigma_k\theta_j[\beta]}{} = \tau_0^{(j)}\cdots\tau_{j-1}^{(j)}\sigma_{k+j} = \theta_j[\beta]\sigma_{k+j}.
\end{gather*}

(ii) Let $\upsilon_r^s = j + \beta +s -r$ for $r$ and $s$ positive integers. The family $\{(\upsilon_r^s)\}_{r,s=1}^j$ satisfies the hypothesis of Lemma \ref{lem:sym} and therefore $\mu_1\cdots\mu_j = \lambda_1\cdots\lambda_j$, where \begin{gather*}
\mu_i = \sigma_{j+\beta+i-1}\cdots\sigma_{\beta+i} \qquad\text{and}\qquad \lambda_i = \sigma_{j+\beta-i+1}\cdots\sigma_{2j+\beta-i}.
\end{gather*} Since $\mu_i = \tau_{i-1}^{(j)}$, we see that $\theta_j[\beta] = \lambda_1\cdots\lambda_j$. As we saw in item~(i), we have that \begin{gather*}
\lambda_{j-i}\sigma_{k+i} = \sigma_{k+i+1}\lambda_{j-i},\qquad 0\leq i\leq j-1.\tag*{\qed}
\end{gather*}\renewcommand{\qed}{}
\end{proof}

\begin{Remark}The intuition behind Lemma \ref{lem:comb} is that the element $\theta_j[\beta]$ exchanges braids in the interval between strands $\beta+1$ and $\beta+j$ with braids in the interval between strands $\beta+j+1$ and $\beta+2j$.

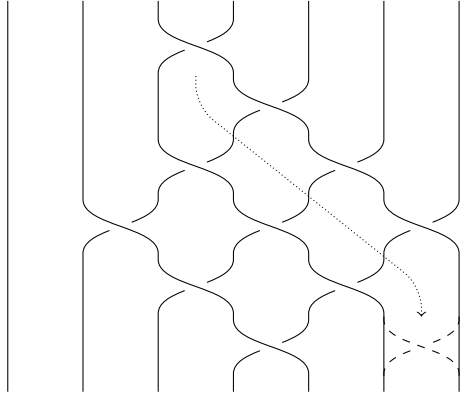
\begin{figure}[h!]\centering
\begin{tikzpicture}[yscale=0.8]
\braid
s_3

s_{4}
s_{3}-s_{5}
s_{2}-s_{4}-s_{6}
s_{3}-s_{5}
s_{4};
\braid[border height=0cm, style strands={1,2}{dashed}] at (6,-5.25)
s_1;

\draw[->, rounded corners=8pt, densely dotted] (3.5,-1.25) -- (3.5,-1.75) -- (6.5,-4.75) -- (6.5,-5.25);
\end{tikzpicture}
\caption{The element $\theta_3[1]$ exchanges the elementary braids $\sigma_3$ and $\sigma_6$.}\label{fig:slidet}
\end{figure}
\end{Remark}

Our next step is to prove that the product does not depend on the chosen representatives.

\begin{Proposition} \label{prop:rep}Let $p'$ and $q'$ be other two representatives of $\mathfrak{p}$ and $\mathfrak{q}$ respectively. Consider the sequence \begin{gather*}
 \mathfrak{r'}_j = B_\infty[\alpha] p'\theta_j[\beta]q' B_\infty[\gamma].
\end{gather*} Then there exists an integer $N>0$ such that \begin{gather*}
\mathfrak{r'}_j = \mathfrak{r}_j, \qquad \text{for all}\quad j\geq N.
\end{gather*}\end{Proposition}

\begin{proof}
Since $p$ and $p'$ are representatives of the same double coset, there exist $r\in B_\infty[\alpha]$ and $h\in B_\infty[\beta]$ such that $p' = rph$. In a similar way, there exist $k\in B_\infty[\beta]$ and $s\in B_\infty[\gamma]$ such that $q' = kqs$. Therefore, \begin{gather*}
 \mathfrak{r'}_j = B_\infty[\alpha] p'\theta_j[\beta]q' B_\infty[\gamma] = B_\infty[\alpha] rph\theta_j[\beta]kqs B_\infty[\gamma] =B_\infty[\alpha] ph\theta_j[\beta]kq B_\infty[\gamma].
 \end{gather*} Consider $N = \max\{\spp p,\spp q,\spp h,\spp k,\alpha,\gamma\}+1$. Given $j\geq N$, let $\bar{h} = C_j\big(h^{-1}\big)$ and $\bar{k} = C_j\big(k^{-1}\big)$. Then $\bar{h},\bar{k}\in B_\infty[j+\beta]$ and hence $\bar{h}\in B_\infty[\gamma]$ and $\bar{k}\in B_\infty[\alpha]$. Furthermore, $\bar{h}$ commutes with $q$ and $k$, and $\bar{k}$ commutes with $p$ and $h$. Now,
 \begin{gather*}
 \bar{k}ph\theta_j[\beta]kq\bar{h} = ph\bar{k}\theta_j[\beta]k\bar{h}q = ph\bar{k}C_j(k)\theta_j[\beta]\bar{h}q \\
 \hphantom{\bar{k}ph\theta_j[\beta]kq\bar{h}}{} =phC_j(k^{-1})C_j(k)\theta_j[\beta]\bar{h}q = ph\theta_j[\beta]\bar{h}q = p\theta_j[\beta]C_j(h)\bar{h}q = p\theta_j[\beta]q.\tag*{\qed}
\end{gather*}\renewcommand{\qed}{}
\end{proof}

Therefore, for all pairs $(\mathfrak{p},\mathfrak{q})\in B_\infty[\alpha]\backslash B_\infty \slash B_\infty[\beta] \times B_\infty[\beta]\backslash B_\infty \slash B_\infty[\gamma]$ we have a well-defined product $\mathfrak{p}\circ\mathfrak{q}\in B_\infty[\alpha]\backslash B_\infty \slash B_\infty[\gamma]$ given by \begin{gather*}
\mathfrak{p}\circ\mathfrak{q} = B_\infty[\alpha] p\theta_j[\beta]q B_\infty[\gamma],
\end{gather*} $p\in\mathfrak{p}$, $q\in\mathfrak{q}$ and $j$ sufficiently large.

Finally, we are going to prove the associativity of the operation $\circ$.

\begin{Proposition}\label{th:asso} The product of double cosets is associative.\end{Proposition}

\begin{proof}Let $\alpha, \beta,\gamma,\delta \in\mathbb{N}$ and consider $\mathfrak{a}\in B_\infty[\alpha]\backslash B_\infty\slash B_\infty[\beta]$, $\mathfrak{b}\in B_\infty[\beta]\backslash B_\infty\slash B_\infty[\gamma]$ and $\mathfrak{c}\in B_\infty[\gamma]\backslash B_\infty\slash B_\infty[\delta]$. Choose representatives $a\in\mathfrak{a}$, $ b\in\mathfrak{b}$ and $c\in\mathfrak{c}$ and consider $k=\max\{\alpha,\beta,\gamma,\delta,\spp a,\spp b,\spp c\}+1$. Then \begin{gather*}
\mathfrak{(ab)c} = B_\infty[\alpha] a\theta_k[\beta]b\theta_l[\gamma]c B_\infty[\delta] \qquad\text{and}\qquad \mathfrak{a(bc)}=B_\infty[\alpha] a\theta_{l'}[\beta]b\theta_k[\gamma]c B_\infty[\delta].
\end{gather*}

To prove our claim we are going to show that the double cosets above are the same, by exhibiting two representatives that are equal (Figs.~\ref{fig:dyn1} and~\ref{fig:dyn2} give an example of the process involved). Here we are assuming $\beta\leq\gamma$, the case $\gamma<\beta$ is analogous.

Throughout the rest of the proof we will use the symbol $a\equiv b$ to signify that $a$ and $b$ are representatives of the same double coset of $B_\infty[\alpha]\backslash B_\infty\slash B_\infty[\gamma]$, that is, we can find elements $h\in B_\infty[\alpha]$ and $k\in B_\infty[\gamma]$ such that $hak=b$.

Using the notation of Lemma~\ref{lem:sym} we can write \begin{gather*}
a\theta_k[\beta]b\theta_l[\gamma]c =a\begin{bmatrix}
 k+\beta&\rightarrow&\beta+1\\
 \downarrow& &\downarrow\\
 2k+\beta-1&\rightarrow&k+\beta
 \end{bmatrix}b\begin{bmatrix}
 2k+\beta+\gamma&\rightarrow&\gamma+1\\
 \downarrow& &\downarrow\\
 4k+2\beta+\gamma-1&\rightarrow&2k+\gamma+\beta
 \end{bmatrix}c, \\
 a\theta_{l'}[\beta]b\theta_k[\gamma]c = a\begin{bmatrix}
 2k+\gamma+\beta&\rightarrow&\beta+1\\
 \downarrow& &\downarrow\\
 4k+2\gamma+\beta-1&\rightarrow&2k+\gamma+\beta
 \end{bmatrix}b\begin{bmatrix}
 k+\gamma&\rightarrow&\gamma+1\\
 \downarrow& &\downarrow\\
 2k+\gamma-1&\rightarrow&k+\gamma
 \end{bmatrix}c.
\end{gather*}
 Using the same lemma, we can see that\begin{align*}
 \theta_k[\beta]&=R_1P;&&P = \begin{bmatrix} k+1&\!\!\rightarrow\!\!&\beta+1\\ \downarrow& &\downarrow\\ 2k&\!\!\rightarrow\!\!&k+\beta \end{bmatrix}\!\!\!,
 &&R_1=\begin{bmatrix} k+\beta&\!\!\rightarrow\!\!&k+2\\ \downarrow& &\downarrow\\ 2k+\beta-1&\!\!\rightarrow\!\!&2k+1 \end{bmatrix}\!\!\!,\\
 \theta_{l'}[\beta]&=R_2P_2;&&P_2 =\begin{bmatrix} k+1&\!\!\rightarrow\!\!&\beta+1\\ \downarrow& &\downarrow\\ 3k+\gamma&\!\!\rightarrow\!\!&2k+\beta+\gamma \end{bmatrix}\!\!\!,
 &&R_2=\begin{bmatrix} 2k+\beta+\gamma&\!\!\rightarrow\!\!&k+2\\ \downarrow& &\downarrow\\ 4k+2\gamma+\beta-1&\!\!\rightarrow\!\!&3k+\gamma+1 \end{bmatrix}\!\!\!,\\
 \theta_k[\gamma]&=P_3R_3;&&P_3 = \begin{bmatrix} k+\gamma&\!\!\rightarrow\!\!&\gamma+1\\ \downarrow& &\downarrow\\ 2k&\!\!\rightarrow\!\!&k+1 \end{bmatrix}\!\!\!,
 &&R_3=\begin{bmatrix} 2k+1&\!\!\rightarrow\!\!&k+2\\ \downarrow& &\downarrow\\ 2k+\gamma-1&\!\!\rightarrow\!\!&k+\gamma \end{bmatrix}\!\!\!,\\
 \theta_l[\gamma]&=P_4R_4;&&P_4 = \begin{bmatrix} 2k+\beta+\gamma&\!\!\rightarrow\!\!&\gamma+1\\ \downarrow& &\downarrow\\ 3k+\beta&\!\!\rightarrow\!\!&k+1 \end{bmatrix}\!\!\!,
 &&R_4=\begin{bmatrix} 3k+\beta&\!\!\rightarrow\!\!&k+2\\ \downarrow& &\downarrow\\ 4k+2\beta+\gamma-1&\!\!\rightarrow\!\!&2k+\beta+\gamma \end{bmatrix}\!\!\!.
 \end{align*}
Since $R_i\in B_\infty[{k+1}]$, $1\leq i \leq 4$, we have\begin{gather*}
aR_1PbP_4R_4c=R_1aPbP_4cR_4 \equiv aPbP_4c,\qquad aR_2P_2bP_3R_3c=R_2aP_2bP_3cR_3 \equiv aP_2bP_3c.
\end{gather*}
Notice also that $P_4=R_5W$, where \begin{gather*} R_5 = \begin{bmatrix} 2k+\beta+\gamma&\rightarrow&2k+2\\ \downarrow& &\downarrow\\ 3k+\beta&\rightarrow&3k-\gamma+2 \end{bmatrix}
 \qquad\text{and}\qquad W=\begin{bmatrix} 2k+1&\rightarrow&\gamma+1\\ \downarrow& &\downarrow\\ 3k-\gamma+1&\rightarrow&k+1 \end{bmatrix}.
 \end{gather*}
Since $\spp P = 2k$ and $R_5\in B_\infty[{2k+1}]$, $R_5P=PR_5$ and we have $aPbR_5Wc=aPR_5bWc = R_5aPbWc\equiv aPbWc$.

Our next objective is to find elements $E,A\in B_\infty$ such that $aP_2bP_3c\equiv aPbEAWc$.

{\it Step 1.} $aP_2bP_3c=aPbELP_3c$. Consider the element\begin{gather*}
 F =\begin{bmatrix} 2k+1&\rightarrow&k+\beta+1\\ \downarrow& &\downarrow\\ 3k+\gamma&\rightarrow&2k+\beta+\gamma \end{bmatrix},
\end{gather*}and notice that $P_2=PF$. Since $F\in B_\infty[{k}]$, we see that $bF=Fb$.

 Moreover, $F=EL$ where \begin{gather*}
 E =\begin{bmatrix} 2k+1&\rightarrow &k+\beta+1\\ \downarrow& &\downarrow\\ 2k-\beta+\gamma&\rightarrow&k+\gamma\end{bmatrix}\qquad\text{and}\qquad L=\begin{bmatrix} 2k-\beta+\gamma+1&\rightarrow&k+\gamma+1\\ \downarrow& &\downarrow\\ 3k+\gamma&\rightarrow&2k+\beta+\gamma \end{bmatrix}.
\end{gather*}

{\it Step 2.} $LP_3c \equiv CP_3c$ for some $C$. In fact, consider \begin{gather*}
C =\begin{bmatrix} 2k+\gamma-\beta+1&\rightarrow &k+\gamma+1\\ \downarrow& &\downarrow\\ 3k-\beta+1&\rightarrow&2k+1\end{bmatrix}\qquad\text{and}\qquad D=\begin{bmatrix} 3k-\beta+2&\rightarrow&2k+2\\ \downarrow& &\downarrow\\ 3k+\gamma&\rightarrow&2k+\beta+\gamma \end{bmatrix}.
\end{gather*}Then $L=CD$ and, since $D\in B_\infty[{2k+1}]$ and $\spp P_3 = 2k$, we have $DP_3c = P_3cD\equiv P_3c$. Hence $LP_3c \equiv CP_3c$.

{\it Step 3.} $CP_3=AW$ for some $A$. In fact, consider $A=\begin{bmatrix} 2k+\gamma-\beta+1&\rightarrow &2k+2\\ \downarrow& &\downarrow\\ 3k-\beta+1&\rightarrow&3k-\gamma+2\end{bmatrix}$. Then,
\begin{gather*}
CP_3 = \begin{bmatrix} 2k+\gamma-\beta+1&\rightarrow &k+\gamma+1\\ \downarrow& &\downarrow\\ 3k-\beta+1&\rightarrow&2k+1\end{bmatrix}\begin{bmatrix} k+\gamma&\rightarrow&\gamma+1\\ \downarrow& &\downarrow\\ 2k&\rightarrow&k+1 \end{bmatrix} \\
\hphantom{CP_3}{}
 =\begin{bmatrix} 2k+\gamma-\beta+1&\!\!\rightarrow\!\! &2k+2\\ \downarrow& &\downarrow\\ 3k-\beta+1&\!\!\rightarrow\!\!&3k-\gamma+2\end{bmatrix}
\begin{bmatrix} 2k+1&\!\!\rightarrow\!\! &k+\gamma+1\\ \downarrow& &\downarrow\\ 3k-\gamma+1&\!\!\rightarrow\!\!&2k+1\end{bmatrix}
\begin{bmatrix} k+\gamma&\!\!\rightarrow\!\!&\gamma+1\\ \downarrow& &\downarrow\\ 2k&\!\!\rightarrow\!\!&k+1 \end{bmatrix}\\
\hphantom{CP_3}{}
=\begin{bmatrix} 2k+\gamma-\beta+1&\rightarrow &2k+2\\ \downarrow& &\downarrow\\ 3k-\beta+1&\rightarrow&3k-\gamma+2\end{bmatrix}\begin{bmatrix} 2k+1&\rightarrow&\gamma+1\\ \downarrow& &\downarrow\\ 3k-\gamma+1&\rightarrow&k+1 \end{bmatrix}=AW.
\end{gather*}

 Therefore, $aP_2bP_3c=aPbELP_3c\equiv aPbECP_3c=aPbEAWc$. At last, consider \begin{gather*}
\tilde{W}=\begin{bmatrix} 3k-\beta+2&\rightarrow &k+2\\ \downarrow& &\downarrow\\ 4k-2\beta+\gamma+1&\rightarrow&2k-\beta+\gamma+1\end{bmatrix}.
\end{gather*}Then $AW\tilde{W}=\theta_{r}[\gamma]$ with $r=2k-\beta+1$. Hence $aPbEAWc\equiv aPbE\theta_{r}[\gamma]c$ and, by Lemma \ref{lem:comb}, $E\theta_{r}[\gamma] = \theta_{r}[\gamma]C_{r}(E)$. Therefore, \begin{gather*}
aPbE\theta_{r}[\gamma]c = aPb\theta_{r}[\gamma]C_{r}(E)c = aPb\theta_{r}[\gamma]cC_{r}(E)\equiv aPb\theta_{r}[\gamma]c \equiv aPbAWc.
\end{gather*} Furthermore, since $A\in B_\infty[{2k+1}]$ and $\spp P = 2k$, \begin{gather*}
aPbAWc = aPAbWc = AaPbWc \equiv aPbWc.\tag*{\qed}
\end{gather*}\renewcommand{\qed}{}
\end{proof}

\begin{Example} In this example we illustrate the method described in the proof of the theorem above. Here we used $a=\sigma_2^{-1} \sigma_1^{-1}$, $b=\sigma_1^2$, $c=\sigma_1^2 \sigma_2^2$, $\alpha=\delta=3$, $\beta=1$ and $\gamma=2$. In each of the figures below, the diagrams are different representatives of the same double coset, obtained following the steps of the proof of Proposition~\ref{theorem:opbr}. The dashed horizontal lines highlight the different braids mentioned in the captions.

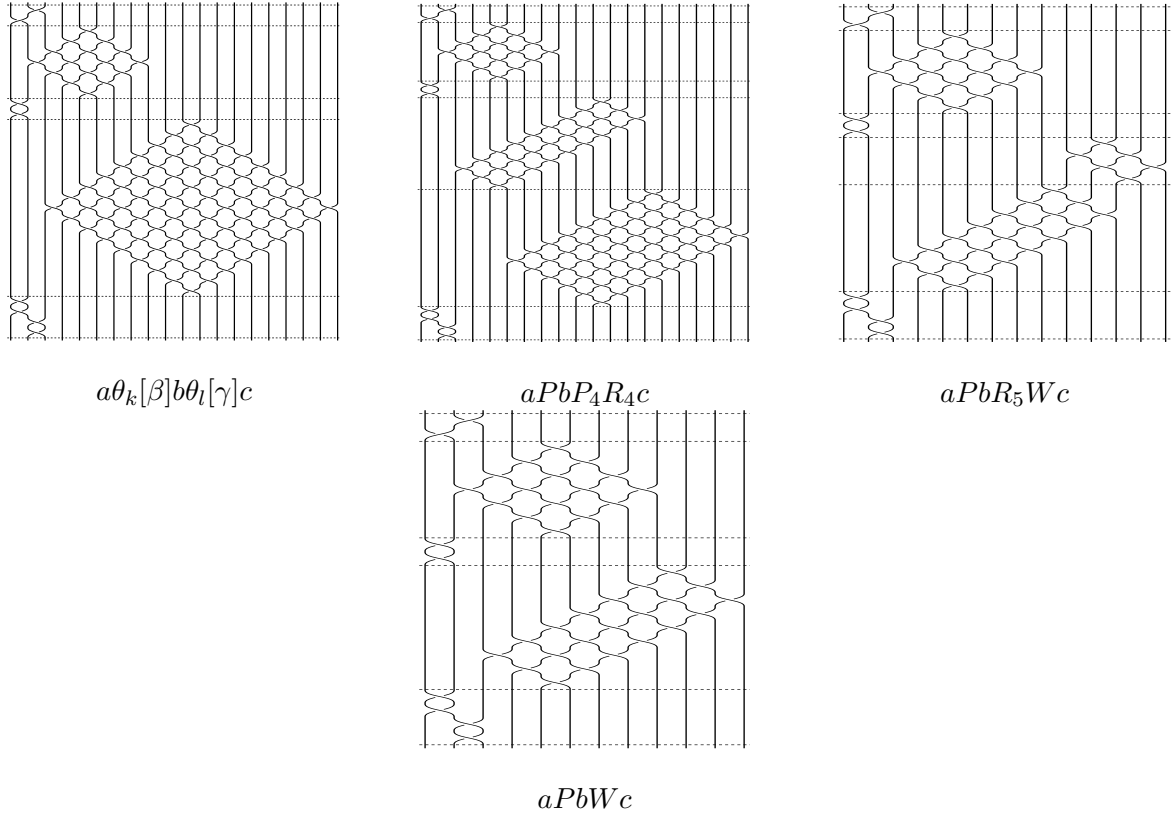
\begin{figure}[h!]\centering
\noindent\mbox{%
\begin{minipage}{4.5cm}
\resizebox{4.5cm}{4.5cm}{
\begin{tikzpicture}%[yscale=.2, xscale=.5]
%\draw[dotted] (0,0) -- (20.2,0);
\draw[dashed] (0.8,-.250) -- (20.2,-.250);
%\draw (0.5, -1.25) node{$a$}; 
\draw[dashed] (0.8,-2.25) -- (20.2,-2.25);
%\draw (0.5, -5.75) node{$\theta_4[1]$}; 
\draw[dashed] (0.8,-9.25) -- (20.2,-9.25);
%\draw (0.5, -10.25) node{$b$}; 
\draw[dashed] (0.8,-11.25) -- (20.2,-11.25);
%\draw (0.5, -19.75) node{$\theta_9[2]$}; 
\draw[dashed] (0.8,-28.25) -- (20.2,-28.25);
%\draw (0.5, -30.25) node{$c$}; 
\draw[dashed] (0.8,-32.25) -- (20.2,-32.25);
\braid
s_2^{-1} s_1^{-1}

s_{5}
s_{4}-s_{6}
s_{3}-s_{5}-s_{7}
s_{2}-s_{4}-s_{6}-s_{8}%
s_{3}-s_{5}-s_{7}
s_{4}-s_{6}
s_{5}

s_1 s_1
s_{11}
s_{10}-s_{12}
s_{9}-s_{11}-s_{13}
s_{8}-s_{10}-s_{12}-s_{14}
s_{7}-s_{9}-s_{11}-s_{13}-s_{15}
s_{6}-s_{8}-s_{10} -s_{12}-s_{14}-s_{16}
s_{5}-s_{7}-s_{9} -s_{11}-s_{13}-s_{15}-s_{17}
s_{4}-s_{6}-s_{8} -s_{10}-s_{12}-s_{14}-s_{16}-s_{18}
s_{3}-s_{5}-s_{7}-s_{9}-s_{11}-s_{13}-s_{15}-s_{17}-s_{19}%
s_{4}-s_{6}-s_{8}-s_{10}-s_{12}-s_{14}-s_{16}-s_{18}
s_{5}-s_{7}-s_{9}-s_{11}-s_{13}-s_{15}-s_{17}
s_{6}-s_{8}-s_{10}-s_{12}-s_{14}-s_{16}
s_{7}-s_{9}-s_{11}-s_{13}-s_{15}
s_{8}-s_{10}-s_{12}-s_{14}
s_{9}-s_{11}-s_{13}
s_{10}-s_{12}
s_{11}

s_1 s_1 s_2 s_2
;

\end{tikzpicture}}%
\begin{center}
$a\theta_k[\beta]b\theta_l[\gamma]c$
\end{center}
\end{minipage}}

\hspace{.7cm}%
\noindent\mbox{%
\begin{minipage}{4.5cm}
\resizebox{4.5cm}{4.5cm}{
\begin{tikzpicture}%[yscale=.2, xscale=.5]
%\draw[dotted] (0,0) -- (20.2,0);
\draw[dashed] (0.8,-.250) -- (20.2,-.250);
%\draw (0.5, -1.25) node{$a$}; 
\draw[dashed] (0.8,-2.25) -- (20.2,-2.25);
%\draw (0.5, -5.75) node{$\theta_4[1]$}; 
\draw[dashed] (0.8,-9.25) -- (20.2,-9.25);
%\draw (0.5, -10.25) node{$b$}; 
\draw[dashed] (0.8,-11.25) -- (20.2,-11.25);
%\draw (0.5, -16.75) node{$P_4$}; 
\draw[dashed] (0.8,-22.25) -- (20.2,-22.25);
%\draw (0.5, -29.25) node{$R_4$}; 
\draw[dashed] (0.8,-36.25) -- (20.2,-36.25);
%\draw (0.5, -38.25) node{$c$}; 
\draw[dashed] (0.8,-40.25) -- (20.2,-40.25);

\braid
s_2^{-1} s_1^{-1}

s_{5}
s_{4}-s_{6}
s_{3}-s_{5}-s_{7}
s_{2}-s_{4}-s_{6}-s_{8}%
s_{3}-s_{5}-s_{7}
s_{4}-s_{6}
s_{5}

s_1 s_1
s_{11}
s_{10}-s_{12}
s_{9}-s_{11}-s_{13}
s_{8}-s_{10}-s_{12}
s_{7}-s_{9}-s_{11}
s_{6}-s_{8}-s_{10} 
s_{5}-s_{7}-s_{9} 
s_{4}-s_{6}-s_{8} 
s_{3}-s_{5}-s_{7} 
s_{4}-s_{6}
s_{5}

%%R_4
s_{14}
s_{13}-s_{15}
s_{12}-s_{14}-s_{16}
s_{11}-s_{13}-s_{15}-s_{17}
s_{10}-s_{12}-s_{14}-s_{16}-s_{18}
s_{9}-s_{11}-s_{13}-s_{15}-s_{17}-s_{19}%
s_{8}-s_{10}-s_{12}-s_{14}-s_{16}-s_{18}
s_{7}-s_{9}-s_{11}-s_{13}-s_{15}-s_{17}
s_{6}-s_{8}-s_{10}-s_{12}-s_{14}-s_{16}
s_{7}-s_{9}-s_{11}-s_{13}-s_{15}
s_{8}-s_{10}-s_{12}-s_{14}
s_{9}-s_{11}-s_{13}
s_{10}-s_{12}
s_{11}

s_1 s_1 s_2 s_2
;

\end{tikzpicture}}%
\begin{center}
$aPbP_4R_4c$
\end{center}
\end{minipage}}

\hspace{.7cm}
\noindent\mbox{%
\begin{minipage}{4.5cm}
\resizebox{4.5cm}{4.5cm}{
\begin{tikzpicture}%[yscale=.2, xscale=.5]
%\draw[dotted][dashed] (0,0) -- (14.2,0);
\draw[dashed] (0.8,-.250) -- (14.2,-.250);
%\draw[dotted] (0.5, -1.25) node{$a$}; 
\draw[dashed] (0.8,-2.25) -- (14.2,-2.25);
%\draw (0.5, -5.75) node{$\theta_4[1]$}; 
\draw[dashed] (0.8,-9.25) -- (14.2,-9.25);
%\draw (0.5, -10.25) node{$b$}; 
\draw[dashed] (0.8,-11.25) -- (14.2,-11.25);
%\draw (0.5, -13.75) node{$R_5$}; 
\draw[dashed] (0.8,-15.25) -- (14.2,-15.25);
%\draw (0.5, -19.25) node{$W$}; 
\draw[dashed] (0.8,-24.25) -- (14.2,-24.25);
%\draw (0.5, -26.25) node{$c$}; 
\draw[dashed] (0.8,-28.25) -- (14.2,-28.25);
\braid
s_2^{-1} s_1^{-1}

s_{5}
s_{4}-s_{6}
s_{3}-s_{5}-s_{7}
s_{2}-s_{4}-s_{6}-s_{8}%t
s_{3}-s_{5}-s_{7}
s_{4}-s_{6}
s_{5}

s_1 s_1
s_{11}
s_{10}-s_{12}
s_{11}-s_{13}
s_{12}
%%%W 
s_{9}
s_{8}-s_{10}
s_{7}-s_{9}-s_{11} 
s_{6}-s_{8}-s_{10} 
s_{5}-s_{7}-s_{9} 
s_{4}-s_{6}-s_{8} 
s_{3}-s_{5}-s_{7} 
s_{4}-s_{6}
s_{5}

s_1 s_1 s_2 s_2
;

\end{tikzpicture}}%
\begin{center}
$aPbR_5Wc$
\end{center}
\end{minipage}}

\hspace{.7cm}
\noindent\mbox{%
\begin{minipage}{4.5cm}
\resizebox{4.5cm}{4.5cm}{
\begin{tikzpicture}%[scale=.5]
%\draw[dotted] (0,0) -- (12.2,0);
\draw[dashed] (0.8,-.250) -- (12.2,-.250);
%\draw (0.5, -1.25) node{$a$}; 
\draw[dashed] (0.8,-2.25) -- (12.2,-2.25);
%\draw (0.5, -5.75) node{$\theta_4[1]$}; 
\draw[dashed] (0.8,-9.25) -- (12.2,-9.25);
%\draw (0.5, -10.25) node{$b$}; 
\draw[dashed] (0.8,-11.25) -- (12.2,-11.25);
%\draw (0.5, -15.75) node{$W$}; 
\draw[dashed] (0.8,-20.25) -- (12.2,-20.25);
%\draw (0.5, -22.25) node{$c$}; 
\draw[dashed] (0.8,-24.25) -- (12.2,-24.25);
\braid
s_2^{-1} s_1^{-1}

s_{5}
s_{4}-s_{6}
s_{3}-s_{5}-s_{7}
s_{2}-s_{4}-s_{6}-s_{8}%
s_{3}-s_{5}-s_{7}
s_{4}-s_{6}
s_{5}

s_1 s_1
%
%%%W 
s_{9}
s_{8}-s_{10}
s_{7}-s_{9}-s_{11} 
s_{6}-s_{8}-s_{10} 
s_{5}-s_{7}-s_{9} 
s_{4}-s_{6}-s_{8} 
s_{3}-s_{5}-s_{7} 
s_{4}-s_{6}
s_{5}

s_1 s_1 s_2 s_2
;

\end{tikzpicture}}%
\begin{center}
$aPbWc$
\end{center}
\end{minipage}}

\caption{The equality $a\theta_k[\beta]b\theta_l[\gamma]c=aPbWc$.}\label{fig:dyn1}
\end{figure}

\begin{figure}[h!]\centering
\noindent\mbox{%
\begin{minipage}{4.5cm}
\resizebox{4.5cm}{4.5cm}{
\begin{tikzpicture}%[yscale=.2, xscale=.5]
%\draw[dotted] (0,0) -- (21.2,0);
\draw[dashed] (0.8,-.250) -- (21.2,-.250);
%\draw (0.5, -1.25) node{$a$}; 
\draw[dashed] (0.8,-2.25) -- (21.2,-2.25);
%\draw (0.5, -11.75) node{$\theta_{10}[1]$}; 
\draw[dashed] (0.8,-21.25) -- (21.2,-21.25);
%\draw (0.5, -22.25) node{$b$}; 
\draw[dashed] (0.8,-23.25) -- (21.2,-23.25);
%\draw (0.5, -26.75) node{$\theta_4[2]$}; 
\draw[dashed] (0.8,-30.25) -- (21.2,-30.25);
%\draw (0.5, -32.25) node{$c$}; 
\draw[dashed] (0.8,-34.25) -- (21.2,-34.25);
\braid
s_2^{-1} s_1^{-1}%a

%%R_2
s_{11}
s_{10}-s_{12}
s_{9}-s_{11}-s_{13}
s_{8}-s_{10}-s_{12}-s_{14}
s_{7}-s_{9}-s_{11}-s_{13}-s_{15}
s_{6}-s_{8}-s_{10}-s_{12}-s_{14}-s_{16}
s_{5}-s_{7}-s_{9}-s_{11}-s_{13}-s_{15}-s_{17}
s_{4}-s_{6}-s_{8}-s_{10}-s_{12}-s_{14}-s_{16}-s_{18}
s_{3}-s_{5}-s_{7}-s_{9}-s_{11}-s_{13}-s_{15}-s_{17}-s_{19}
s_{2}-s_{4}-s_{6}-s_{8}-s_{10}-s_{12}-s_{14}-s_{16}-s_{18}-s_{20}
s_{3}-s_{5}-s_{7}-s_{9}-s_{11}-s_{13}-s_{15}-s_{17}-s_{19}
s_{4}-s_{6}-s_{8}-s_{10}-s_{12}-s_{14}-s_{16}-s_{18}
s_{5}-s_{7}-s_{9}-s_{11}-s_{13}-s_{15}-s_{17}
s_{6}-s_{8}-s_{10}-s_{12}-s_{14}-s_{16}
s_{7}-s_{9}-s_{11}-s_{13}-s_{15}
s_{8}-s_{10}-s_{12}-s_{14}
s_{9}-s_{11}-s_{13}
s_{10}-s_{12}
s_{11}

s_1 s_1%

%%P_3
s_{6}
s_{5}-s_{7}
s_{4}-s_{6}-s_{8}
s_{3}-s_{5}-s_{7}-s_{9}
s_{4}-s_{6}-s_{8}
s_{5}-s_{7}
s_{6}

s_1 s_1 s_2 s_2%c
;

\end{tikzpicture}}%
\begin{center}
$a\theta_{l'}[\beta]b\theta_k[\gamma]c$
\end{center}
\end{minipage}}

\hspace{.7cm}
\noindent\mbox{%
\begin{minipage}{4.5cm}
\resizebox{4.5cm}{4.5cm}{
\begin{tikzpicture}%[yscale=.2, xscale=.5]
%\draw[dotted] (0,0) -- (21.2,0);
\draw[dashed] (0.8,-.250) -- (21.2,-.250);
%\draw (0.5, -1.25) node{$a$}; 
\draw[dashed] (0.8,-2.25) -- (21.2,-2.25);
%\draw (0.5, -9.75) node{$R_2$}; 
\draw[dashed] (0.8,-17.25) -- (21.2,-17.25);
%\draw (0.5, -23.25) node{$P_2$}; 
\draw[dashed] (0.8,-30.25) -- (21.2,-30.25);
%\draw (0.5, -31.75) node{$b$}; 
\draw[dashed] (0.8,-32.25) -- (21.2,-32.25);
%\draw (0.5, -36.25) node{$P_3$}; 
\draw[dashed] (0.8,-38.25) -- (21.2,-38.25);
%\draw (0.5, -40.25) node{$R_3$}; 
\draw[dashed] (0.8,-42.25) -- (21.2,-42.25);
%\draw (0.5, -44.25) node{$c$}; 
\draw[dashed] (0.8,-46.25) -- (21.2,-46.25);

\braid
s_2^{-1} s_1^{-1}%a

%%R_2
s_{11}
s_{10}-s_{12}
s_{9}-s_{11}-s_{13}
s_{8}-s_{10}-s_{12}-s_{14}
s_{7}-s_{9}-s_{11}-s_{13}-s_{15}
s_{6}-s_{8}-s_{10}-s_{12}-s_{14}-s_{16}
s_{7}-s_{9}-s_{11}-s_{13}-s_{15}-s_{17}
s_{8}-s_{10}-s_{12}-s_{14}-s_{16}-s_{18}
s_{9}-s_{11}-s_{13}-s_{15}-s_{17}-s_{19}
s_{10}-s_{12}-s_{14}-s_{16}-s_{18}-s_{20}
s_{11}-s_{13}-s_{15}-s_{17}-s_{19}
s_{12}-s_{14}-s_{16}-s_{18}
s_{13}-s_{15}-s_{17}
s_{14}-s_{16}
s_{15}
%%P_2
s_{5}
s_{4}-s_{6}
s_{3}-s_{5}-s_{7}
s_{2}-s_{4}-s_{6}-s_{8}
s_{3}-s_{5}-s_{7}-s_{9}
s_{4}-s_{6}-s_{8}-s_{10}
s_{5}-s_{7}-s_{9}-s_{11}
s_{6}-s_{8}-s_{10}-s_{12}
s_{7}-s_{9}-s_{11}-s_{13}
s_{8}-s_{10}-s_{12}-s_{14}
s_{9}-s_{11}-s_{13}
s_{10}-s_{12}
s_{11}

s_1 s_1%b

%%P_3
s_{6}
s_{5}-s_{7}
s_{4}-s_{6}-s_{8}
s_{3}-s_{5}-s_{7}
s_{4}-s_{6}
s_{5}
%%R_3
s_{9}
s_{8}
s_{7}
s_{6}

s_1 s_1 s_2 s_2%c
;

\end{tikzpicture}}%
\begin{center}
$aR_2P_2bP_3R_3c$
\end{center}
\end{minipage}}

\hspace{.7cm}
\noindent\mbox{%
\begin{minipage}{4.5cm}
\resizebox{4.5cm}{4.5cm}{
\begin{tikzpicture}%[yscale=.2, xscale=.5]
%\draw[dotted] (0,0) -- (15.2,0);
\draw[dashed] (0.8,-.250) -- (15.2,-.250);
%\draw (0.5, -1.25) node{$a$}; 
\draw[dashed] (0.8,-2.25) -- (15.2,-2.25);
%\draw (0.5, -5.75) node{$P$}; 
\draw[dashed] (0.8,-9.25) -- (15.2,-9.25);
%\draw (0.5, -10.25) node{$b$}; 
\draw[dashed] (0.8,-11.25) -- (15.2,-11.25);
%\draw (0.5, -15.75) node{$F$}; 
\draw[dashed] (0.8,-20.25) -- (15.2,-20.25);
%\draw (0.5, -23.25) node{$P_3$}; 
\draw[dashed] (0.8,-26.25) -- (15.2,-26.25);
%\draw (0.5, -28.25) node{$c$}; 
\draw[dashed] (0.8,-30.25) -- (15.2,-30.25);

\braid
s_2^{-1} s_1^{-1}%a

%%R_2

%%P_2
%%%P
s_{5}
s_{4}-s_{6}
s_{3}-s_{5}-s_{7}
s_{2}-s_{4}-s_{6}-s_{8}
s_{3}-s_{5}-s_{7}
s_{4}-s_{6}
s_{5}

s_1 s_1%b
%%%F
s_{9}
s_{8}-s_{10}
s_{7}-s_{9}-s_{11}
s_{6}-s_{8}-s_{10}-s_{12}
s_{7}-s_{9}-s_{11}-s_{13}
s_{8}-s_{10}-s_{12}-s_{14}
s_{9}-s_{11}-s_{13}
s_{10}-s_{12}
s_{11}

%%P_3
s_{6}
s_{5}-s_{7}
s_{4}-s_{6}-s_{8}
s_{3}-s_{5}-s_{7}
s_{4}-s_{6}
s_{5}

s_1 s_1 s_2 s_2%c
;

\end{tikzpicture}}%
\begin{center}
$aPbFP_3c$
\end{center}
\end{minipage}}

\hspace{.7cm}
\noindent\mbox{%
\begin{minipage}{4.5cm}
\resizebox{4.5cm}{4.5cm}{
\begin{tikzpicture}%[yscale=.2, xscale=.5]
%\draw[dotted] (0,0) -- (15.2,0);
\draw[dashed] (0.8,-.250) -- (15.2,-.250);
%\draw (0.5, -1.25) node{$a$}; 
\draw[dashed] (0.8,-2.25) -- (15.2,-2.25);
%\draw (0.5, -5.75) node{$P$}; 
\draw[dashed] (0.8,-9.25) -- (15.2,-9.25);
%\draw (0.5, -10.25) node{$b$}; 
\draw[dashed] (0.8,-11.25) -- (15.2,-11.25);
%\draw (0.5, -13.75) node{$E$}; 
\draw[dashed] (0.8,-15.25) -- (15.2,-15.25);
%\draw (0.5, -19.25) node{$L$}; 
\draw[dashed] (0.8,-23.25) -- (15.2,-23.25);
%\draw (0.5, -26.25) node{$P_3$}; 
\draw[dashed] (0.8,-29.25) -- (15.2,-29.25);
%\draw (0.5, -31.25) node{$c$}; 
\draw[dashed] (0.8,-33.25) -- (15.2,-33.25);

\braid
s_2^{-1} s_1^{-1}%a

%
%%R_2

%%P_2
%%%P
s_{5}
s_{4}-s_{6}
s_{3}-s_{5}-s_{7}
s_{2}-s_{4}-s_{6}-s_{8}
s_{3}-s_{5}-s_{7}
s_{4}-s_{6}
s_{5}
%%%F

s_1 s_1%b

%%%%E
s_{9}
s_{8}
s_{7}
s_{6}
%%%%F
s_{10}
s_{9}-s_{11}
s_{8}-s_{10}-s_{12}
s_{7}-s_{9}-s_{11}-s_{13}
s_{8}-s_{10}-s_{12}-s_{14}
s_{9}-s_{11}-s_{13}
s_{10}-s_{12}
s_{11}

%%P_3
s_{6}
s_{5}-s_{7}
s_{4}-s_{6}-s_{8}
s_{3}-s_{5}-s_{7}
s_{4}-s_{6}
s_{5}
%%R_3

s_1 s_1 s_2 s_2%c
;

\end{tikzpicture}}%
\begin{center}
$aPbELP_3c$
\end{center}
\end{minipage}}

\hspace{.7cm}
\noindent\mbox{%
\begin{minipage}{4.5cm}
\resizebox{4.5cm}{4.5cm}{
\begin{tikzpicture}%[yscale=.2, xscale=.5]
%\draw[dotted] (0,0) -- (15.2,0);
\draw[dashed] (0.8,-.250) -- (15.2,-.250);
%\draw (0.5, -1.25) node{$a$}; 
\draw[dashed] (0.8,-2.25) -- (15.2,-2.25);
%\draw (0.5, -5.75) node{$P$}; 
\draw[dashed] (0.8,-9.25) -- (15.2,-9.25);
%\draw (0.5, -10.25) node{$b$}; 
\draw[dashed] (0.8,-11.25) -- (15.2,-11.25);
%\draw (0.5, -13.75) node{$E$}; 
\draw[dashed] (0.8,-15.25) -- (15.2,-15.25);
%\draw (0.5, -18.25) node{$C$}; 
\draw[dashed] (0.8,-21.25) -- (15.2,-21.25);
%\draw (0.5, -23.25) node{$D$}; 
\draw[dashed] (0.8,-26.25) -- (15.2,-26.25);
%\draw (0.5, -29.25) node{$P_3$}; 
\draw[dashed] (0.8,-32.25) -- (15.2,-32.25);
%\draw (0.5, -34.25) node{$c$}; 
\draw[dashed] (0.8,-36.25) -- (15.2,-36.25);

\braid
s_2^{-1} s_1^{-1}%a

%%R_2
%%P_2
%%%P
s_{5}
s_{4}-s_{6}
s_{3}-s_{5}-s_{7}
s_{2}-s_{4}-s_{6}-s_{8}
s_{3}-s_{5}-s_{7}
s_{4}-s_{6}
s_{5}
%%%F

s_1 s_1%b

%%%%E
s_{9}
s_{8}
s_{7}
s_{6}
%%%%L
%%%%%C
s_{10}
s_{9}-s_{11}
s_{8}-s_{10}-s_{12}
s_{7}-s_{9}-s_{11}
s_{8}-s_{10}
s_{9}
%%%%%D
s_{13}
s_{12}-s_{14}
s_{11}-s_{13}
s_{10}-s_{12}
s_{11}

%%P_3
s_{6}
s_{5}-s_{7}
s_{4}-s_{6}-s_{8}
s_{3}-s_{5}-s_{7}
s_{4}-s_{6}
s_{5}
%%R_3

s_1 s_1 s_2 s_2%c
;

\end{tikzpicture}}%
\begin{center}
$aPbECDP_3c$
\end{center}
\end{minipage}}

\hspace{.7cm}
\noindent\mbox{%
\begin{minipage}{4.5cm}
\resizebox{4.5cm}{4.5cm}{
\begin{tikzpicture}%[yscale=.2, xscale=.5]
%\draw[dotted] (0,0) -- (13.2,0);
\draw[dashed] (0.8,-.250) -- (13.2,-.250);
%\draw (0.5, -1.25) node{$a$}; 
\draw[dashed] (0.8,-2.25) -- (13.2,-2.25);
%\draw (0.5, -5.75) node{$P$}; 
\draw[dashed] (0.8,-9.25) -- (13.2,-9.25);
%\draw (0.5, -10.25) node{$b$}; 
\draw[dashed] (0.8,-11.25) -- (13.2,-11.25);
%\draw (0.5, -13.75) node{$E$}; 
\draw[dashed] (0.8,-15.25) -- (13.2,-15.25);
%\draw (0.5, -20.25) node{$AW$}; 
\draw[dashed] (0.8,-25.25) -- (13.2,-25.25);
%\draw (0.5, -27.25) node{$c$}; 
\draw[dashed] (0.8,-29.25) -- (13.2,-29.25);

\braid
s_2^{-1} s_1^{-1}%a

%%R_2

%%P_2
%%%P
s_{5}
s_{4}-s_{6}
s_{3}-s_{5}-s_{7}
s_{2}-s_{4}-s_{6}-s_{8}
s_{3}-s_{5}-s_{7}
s_{4}-s_{6}
s_{5}
%%%F

s_1 s_1%b

%%%%E
s_{9}
s_{8}
s_{7}
s_{6}
%%%%L
%%%%%C

%%%%%D
%s_{13}
%s_{12}-s_{14}
%s_{11}-s_{13}
%s_{10}-s_{12}
%s_{11}

%A
s_{10}
s_{9}-s_{11} 
s_{8}-s_{10}-s_{12}
s_{7}-s_{9}-s_{11}
s_{6}-s_{8}-s_{10}
s_{5}-s_{7}-s_{9}
s_{4}-s_{6}-s_{8}
s_{3}-s_{5}-s_{7}
s_{4}-s_{6}
s_{5}
%%R_3

s_1 s_1 s_2 s_2%c
;

\end{tikzpicture}}%
\begin{center}
$aPbEAWc$
\end{center}
\end{minipage}}

\hspace{.7cm}
\noindent\mbox{%
\begin{minipage}{4.5cm}
\resizebox{4.5cm}{4.5cm}{
\begin{tikzpicture}%[yscale=.2, xscale=.5]
%\draw[dotted] (0,0) -- (18.2,0);
\draw[dashed] (0.8,-.250) -- (18.2,-.250);
%\draw (0.5, -1.25) node{$a$}; 
\draw[dashed] (0.8,-2.25) -- (18.2,-2.25);
%\draw (0.5, -5.75) node{$P$}; 
\draw[dashed] (0.8,-9.25) -- (18.2,-9.25);
%\draw (0.5, -10.25) node{$b$}; 
\draw[dashed] (0.8,-11.25) -- (18.2,-11.25);
%\draw (0.5, -13.75) node{$E$}; 
\draw[dashed] (0.8,-15.25) -- (18.2,-15.25);
%\draw (0.5, -22.25) node{$\theta_8[2]$}; 
\draw[dashed] (0.8,-30.25) -- (18.2,-30.25);
%\draw (0.5, -32.25) node{$c$}; 
\draw[dashed] (0.8,-34.25) -- (18.2,-34.25);
\braid
s_2^{-1} s_1^{-1}%a

%%R_2

%%P_2
%%%P
s_{5}
s_{4}-s_{6}
s_{3}-s_{5}-s_{7}
s_{2}-s_{4}-s_{6}-s_{8}
s_{3}-s_{5}-s_{7}
s_{4}-s_{6}
s_{5}
%%%F

s_1 s_1%b

%%%%E
s_{9}
s_{8}
s_{7}
s_{6}
%%%%L
%%%%%C

%%%%%D
%s_{13}
%s_{12}-s_{14}
%s_{11}-s_{13}
%s_{10}-s_{12}
%s_{11}

%
%A
s_{10}

%W
s_{9}-s_{11}
s_{8}-s_{10}-s_{12}
s_{7}-s_{9}-s_{11}-s_{13}
s_{6}-s_{8}-s_{10}-s_{12}-s_{14}
s_{5}-s_{7}-s_{9}-s_{11}-s_{13}-s_{15}
s_{4}-s_{6}-s_{8}-s_{10}-s_{12}-s_{14}-s_{16}
s_{3}-s_{5}-s_{7}-s_{9}-s_{11}-s_{13}-s_{15}-s_{17}
s_{4}-s_{6}-s_{8}-s_{10}-s_{12}-s_{14}-s_{16}
s_{5}-s_{7}-s_{9}-s_{11}-s_{13}-s_{15}
s_{6}-s_{8}-s_{10}-s_{12}-s_{14}
s_{7}-s_{9}-s_{11}-s_{13}
s_{8}-s_{10}-s_{12}
s_{9}-s_{11}
s_{10}
%%R_3

s_1 s_1 s_2 s_2%c
;

\end{tikzpicture}}%
\begin{center}
$aPbE\theta_t[\gamma]c$
\end{center}
\end{minipage}}

\hspace{.7cm}
\noindent\mbox{%
\begin{minipage}{4.5cm}
\resizebox{4.5cm}{4.5cm}{
\begin{tikzpicture}%[yscale=.2, xscale=.5]
%\draw[dotted] (0,0) -- (18.2,0);
\draw[dashed] (0.8,-.250) -- (18.2,-.250);
%\draw (0.5, -1.25) node{$a$}; 
\draw[dashed] (0.8,-2.25) -- (18.2,-2.25);
%\draw (0.5, -5.75) node{$P$}; 
\draw[dashed] (0.8,-9.25) -- (18.2,-9.25);
%\draw (0.5, -10.25) node{$b$}; 
\draw[dashed] (0.8,-11.25) -- (18.2,-11.25);
%\draw (0.5, -18.75) node{$\theta_8[2]$}; 
\draw[dashed] (0.8,-26.25) -- (18.2,-26.25);
%\draw (0.5, -28.25) node{$C_8(E)$}; 
\draw[dashed] (0.8,-30.25) -- (18.2,-30.25);
%\draw (0.5, -32.25) node{$c$}; 
\draw[dashed] (0.8,-34.25) -- (18.2,-34.25);
\braid
s_2^{-1} s_1^{-1}%a

%%R_2

%%P_2
%%%P
s_{5}
s_{4}-s_{6}
s_{3}-s_{5}-s_{7}
s_{2}-s_{4}-s_{6}-s_{8}
s_{3}-s_{5}-s_{7}
s_{4}-s_{6}
s_{5}
%%%F
%%%%E

%%%%L
%%%%%C

s_1 s_1%b

%A
s_{10}

%W
s_{9}-s_{11}
s_{8}-s_{10}-s_{12}
s_{7}-s_{9}-s_{11}-s_{13}
s_{6}-s_{8}-s_{10}-s_{12}-s_{14}
s_{5}-s_{7}-s_{9}-s_{11}-s_{13}-s_{15}
s_{4}-s_{6}-s_{8}-s_{10}-s_{12}-s_{14}-s_{16}
s_{3}-s_{5}-s_{7}-s_{9}-s_{11}-s_{13}-s_{15}-s_{17}
s_{4}-s_{6}-s_{8}-s_{10}-s_{12}-s_{14}-s_{16}
s_{5}-s_{7}-s_{9}-s_{11}-s_{13}-s_{15}
s_{6}-s_{8}-s_{10}-s_{12}-s_{14}
s_{7}-s_{9}-s_{11}-s_{13}
s_{8}-s_{10}-s_{12}
s_{9}-s_{11}
s_{10}
%%R_3

s_{17}
s_{16}
s_{15}
s_{14}

s_1 s_1 s_2 s_2%c
;

\end{tikzpicture}}%
\begin{center}
$aPb\theta_t[\gamma]C_t(E)c$
\end{center}
\end{minipage}}

\hspace{.7cm}
\noindent\mbox{%
\begin{minipage}{4.5cm}
\resizebox{4.5cm}{4.5cm}{
\begin{tikzpicture}%[yscale=.2, xscale=.5]
%\draw[dotted] (0,0) -- (20.2,0);
\draw[dashed] (0.8,-.250) -- (18.2,-.250);
%\draw (0.5, -1.25) node{$a$}; 
\draw[dashed] (0.8,-2.25) -- (18.2,-2.25);
%\draw (0.5, -5.75) node{$P$}; 
\draw[dashed] (0.8,-9.25) -- (18.2,-9.25);
%\draw (0.5, -10.25) node{$b$}; 
\draw[dashed] (0.8,-11.25) -- (18.2,-11.25);
%\draw (0.5, -12.75) node{$A$}; 
\draw[dashed] (0.8,-14.25) -- (18.2,-14.25);
%\draw (0.5, -18.25) node{$W$}; 
\draw[dashed] (0.8,-23.25) -- (18.2,-23.25);
%\draw (0.5, -29.25) node{$\tilde{W}$}; 
\draw[dashed] (0.8,-35.25) -- (18.2,-35.25);
%\draw (0.5, -37.25) node{$c$}; 
\draw[dashed] (0.8,-39.25) -- (18.2,-39.25);
\braid
s_2^{-1} s_1^{-1}%a

%%P_2
%%%P
s_{5}
s_{4}-s_{6}
s_{3}-s_{5}-s_{7}
s_{2}-s_{4}-s_{6}-s_{8}
s_{3}-s_{5}-s_{7}
s_{4}-s_{6}
s_{5}

s_1 s_1%b

%A
s_{10}
s_{11}
s_{12}
%W
s_{9} 
s_{8}-s_{10} 
s_{7}-s_{9}-s_{11}
s_{6}-s_{8}-s_{10}
s_{5}-s_{7}-s_{9}
s_{4}-s_{6}-s_{8}
s_{3}-s_{5}-s_{7}
s_{4}-s_{6}
s_{5}

s_{13}
s_{12}-s_{14}
s_{11}-s_{13}-s_{15}
s_{10}-s_{12}-s_{14}-s_{16}
s_{9}-s_{11}-s_{13}-s_{15}-s_{17}
s_{8}-s_{10}-s_{12}-s_{14}-s_{16}
s_{7}-s_{9}-s_{11}-s_{13}-s_{15}
s_{6}-s_{8}-s_{10}-s_{12}-s_{14}
s_{7}-s_{9}-s_{11}-s_{13}
s_{8}-s_{10}-s_{12}
s_{9}-s_{11}
s_{10}
%%R_3

s_1 s_1 s_2 s_2%c
;

\end{tikzpicture}}%
\begin{center}
$aPbAW\tilde{W}c$
\end{center}
\end{minipage}}

\hspace{.7cm}
\noindent\mbox{%
\begin{minipage}{4.5cm}
\resizebox{4.5cm}{4.5cm}{
\begin{tikzpicture}%[scale=.5]
%\draw[dotted] (0,0) -- (12.2,0);
\draw[dashed] (0.8,-.250) -- (12.2,-.250);
%\draw (0.5, -1.25) node{$a$}; 
\draw[dashed] (0.8,-2.25) -- (12.2,-2.25);
%\draw (0.5, -5.75) node{$\theta_4[1]$}; 
\draw[dashed] (0.8,-9.25) -- (12.2,-9.25);
%\draw (0.5, -10.25) node{$b$}; 
\draw[dashed] (0.8,-11.25) -- (12.2,-11.25);
%\draw (0.5, -15.75) node{$W$}; 
\draw[dashed] (0.8,-20.25) -- (12.2,-20.25);
%\draw (0.5, -22.25) node{$c$}; 
\draw[dashed] (0.8,-24.25) -- (12.2,-24.25);
\braid
s_2^{-1} s_1^{-1}

s_{5}
s_{4}-s_{6}
s_{3}-s_{5}-s_{7}
s_{2}-s_{4}-s_{6}-s_{8}%
s_{3}-s_{5}-s_{7}
s_{4}-s_{6}
s_{5}

s_1 s_1
%
%%%W 
s_{9}
s_{8}-s_{10}
s_{7}-s_{9}-s_{11} 
s_{6}-s_{8}-s_{10} 
s_{5}-s_{7}-s_{9} 
s_{4}-s_{6}-s_{8} 
s_{3}-s_{5}-s_{7} 
s_{4}-s_{6}
s_{5}

s_1 s_1 s_2 s_2
;

\end{tikzpicture}}%
\begin{center}
$aPbWc$
\end{center}
\end{minipage}}

\caption{The equality $a\theta_{l'}[\beta]b\theta_k[\gamma]c=aPbWc$.}\label{fig:dyn2}
\end{figure}
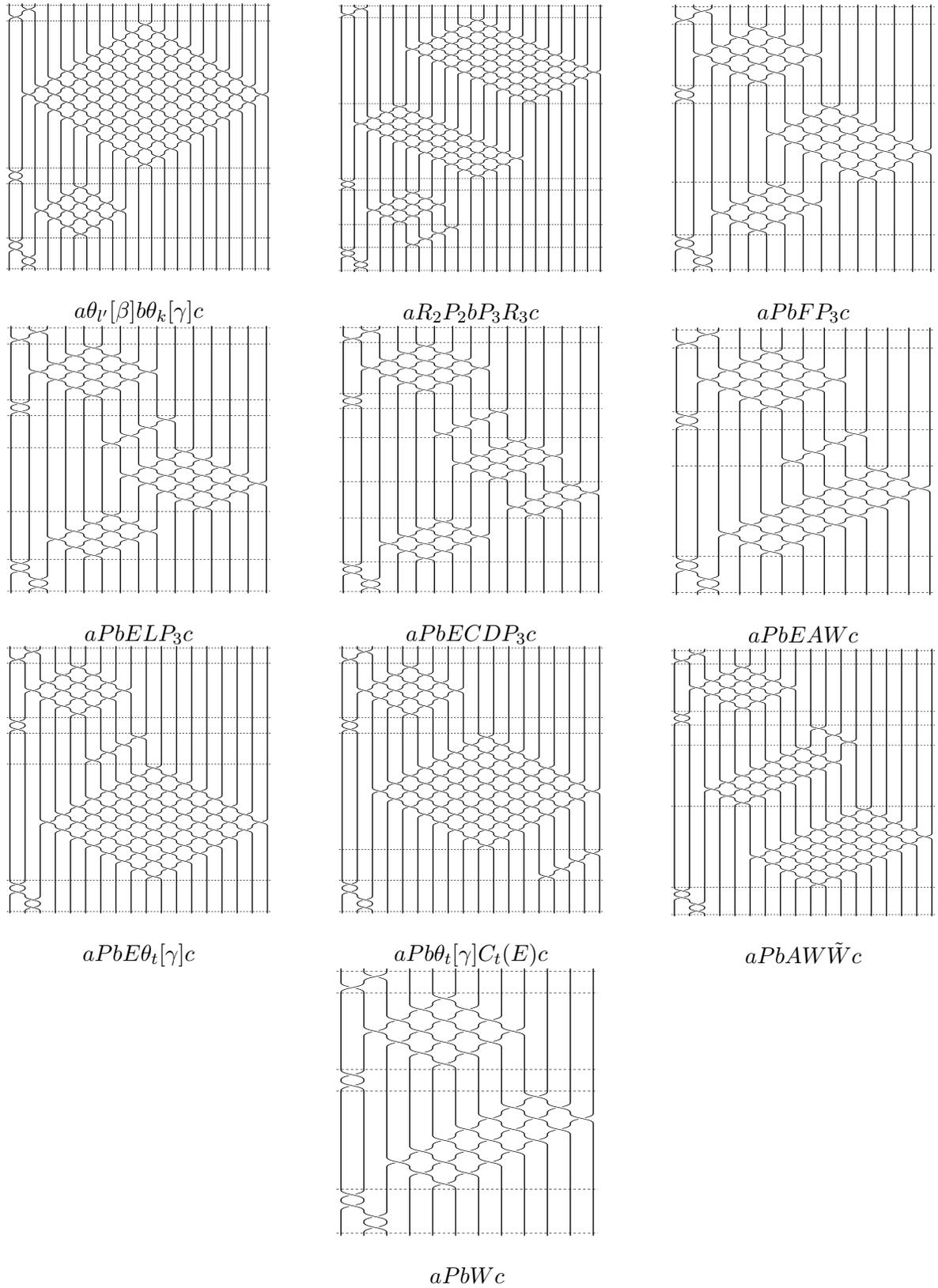
\end{Example}

\subsection{Proof of Proposition \ref{prop:mene}}

We show that the conjugacy problem for $B_\infty$ can be reduced to a conjugacy problem in $B_n$, for some $n$.
Given two braids $x,y\in B_\infty$, since these braids are finitely supported, there exists $n\in\mathbb{N}$ such that we can consider these braids as elements of $B_n$. Since the conjugacy problem has a solution in $B_n$, to prove the proposition it suffices to show that $x$ is conjugate to $y$ in~$B_n$ if and only if they are conjugate in $B_\infty$. But this follows from the properties of the direct limit and the fact that the inclusions $i_n\colon B_n\to B_{n+1}$ do not merge conjugacy classes (see~\cite{meneses}).

\subsection{Proof of Theorem \ref{theorem:opma}}
 Let $p$ and $q$ be representatives of the double cosets $\mathfrak{p}\in G[n]\backslash {\rm GL}(\infty)\slash G[k]$ and $\mathfrak{q}\in G[k]\backslash {\rm GL}(\infty)\slash G[m]$, respectively. Define the sequence of double cosets \begin{gather*}
\mathfrak{r}_j = G[n] p\Theta_j[k]q G[m],
\end{gather*} in $G[n]\backslash {\rm GL}(\infty)\slash G[m]$.

We remark the following identity:

\begin{Lemma}\label{lem:thetas} If $\eta\colon B_\infty\to {\rm GL}(\infty)$ is the Burau representation, as defined in Section~{\rm \ref{sec:burau}}, the following identity holds\begin{gather*}
\Theta_j[k] = \eta(\theta_j[k]), \qquad \text{for all} \quad j,k\in \mathbb{N}.
\end{gather*}
\end{Lemma}

\begin{Proposition} The sequence $\mathfrak{r}_j$ above is eventually constant and its limit does not depend on the choice of representatives. \end{Proposition}
\begin{proof}
Let $N\in \mathbb{N}$ be such that $N>\max\{m,k,n\}$ and $p$ and $q$ can be written as square $(k+N+\infty)$-matrices with the following block configuration: \begin{gather*}
p = \begin{pmatrix}
a&b&0\\
c&d&0\\
0&0&1_\infty
\end{pmatrix},\qquad
q = \begin{pmatrix}
x&y&0\\
z&w&0\\
0&0&1_\infty
\end{pmatrix}.
\end{gather*}Suppose that for some $i\geq N$ we have $\mathfrak{r}_i = \mathfrak{r}_N$. We show that $\mathfrak{r}_i = \mathfrak{r}_{i+1}$. As we saw in Proposition~\ref{prop:suflarg}, there are elements $u,l\in B_\infty$ such that $\theta_i[k] = u\theta_{i+1}[k]l$. Hence, if $U = \eta(u)$ and $L = \eta(l)$, we have \begin{gather*}
\Theta_i[k] = U\Theta_{i+1}[k]L.
\end{gather*} Furthermore, $U$ and $L$ have the following block configuration \begin{gather*}
U = \begin{pmatrix}
1_k&0&0&0\\
0&1_i&0&0\\
0&0&\nu&0\\
0&0&0&1_\infty
\end{pmatrix}, \qquad L = \begin{pmatrix}
1_k&0&0&0\\
0&1_i&0&0\\
0&0&\lambda&0\\
0&0&0&1_\infty
\end{pmatrix}.
\end{gather*} Thus, \begin{gather*}
Up = pU\qquad\text{and}\qquad Lq = qL.
\end{gather*} Consequently, \begin{gather*}
p\Theta_{i}[k]q = pU\Theta_{i+1}[k]Lq = Up\Theta_{i+1}[k]qL.
\end{gather*} Since $U$ and $L$ are elements of the image of the Burau representation $\eta$, we have that $U,L\in G[k]$ and therefore \begin{gather*}
\mathfrak{r}_{i+1} = G[n] p\Theta_{i+1}[k]q G[m] = G[n] Up\Theta_{i+1}[k]qL G[m] = G[n] p\Theta_{i}[k]q G[m] = \mathfrak{r}_i.
\end{gather*}

To show that the limit of the sequence $\mathfrak{r_i}$ does not depend on the choice of representatives it suffices to show that, for any $H$ and $J$ in $G[k]$, we have \begin{gather*}
 \lim G[n] p\Theta_i[k]q G[m] = \lim G[n] pJ\Theta_i[k]Hq G[m].
\end{gather*}
Let $N>0$ be as before. Consider $M>N$ such that $H$ and $J$ are square $(k+M+\infty)$-matrices with the block configuration: \begin{gather*}
H = \begin{pmatrix}
1_k&0&0\\
0&h&0\\
0&0&1_\infty
\end{pmatrix},\qquad
J = \begin{pmatrix}
1_k&0&0\\
0&j&0\\
0&0&1_\infty
\end{pmatrix}.
\end{gather*}Since $H$ preserves the vector $v$, we have that $V_{M}h = V_M$. Similarly, $jV_M = V_M$. Therefore,
\begin{gather*}
J\Theta_M[k]H = \begin{pmatrix}
1_k&0&0&0\\
0&j&0&0\\
0&0&1_M&0\\
0&0&0&1_\infty
\end{pmatrix}
\begin{pmatrix}
1_k&0&0\\
0&V_M&t^M1_M&0\\
0&1_M&0&0\\
0&0&0&1_\infty
\end{pmatrix}
\begin{pmatrix}
1_k&0&0&0\\
0&h&0&0\\
0&0&1_M&0\\
0&0&0&1_\infty
\end{pmatrix}\\
\hphantom{J\Theta_M[k]H} =\begin{pmatrix}
1_k&0&0&0\\
0&jV_Mh&t^Mh&0\\
0&j&0&0\\
0&0&0&1_\infty
\end{pmatrix} =
\begin{pmatrix}
1_k&0&0&0\\
0&V_M&t^Mh&0\\
0&j&0&0\\
0&0&0&1_\infty
\end{pmatrix}\\
\hphantom{J\Theta_M[k]H}
= \begin{pmatrix}
1_k&0&0&0\\
0&1_M&0&0\\
0&0&h&0\\
0&0&0&1_\infty
\end{pmatrix}\begin{pmatrix}
1_k&0&0&0\\
0&V_M&t^M1_M&0\\
0&1_M&0&0\\
0&0&0&1_\infty
\end{pmatrix}\begin{pmatrix}
1_k&0&0&0\\
0&1_M&0&0\\
0&0&j&0\\
0&0&0&1_\infty
\end{pmatrix}.
\end{gather*}
Call $J'$ the new matrix containing the block $j$ and $H'$ the new matrix containing the block $h$. Then we have \begin{gather*}
pJ\Theta_M[k]Hq = pH'\Theta_M[k]J'q = H'p\Theta_M[k]qJ'.
\end{gather*} Therefore, $p\Theta_M[k]q$ and $pJ\Theta_M[k]Hq$ belong to the same double coset for $M$ sufficiently large.\end{proof}

Therefore, we have a well-defined product of the double cosets $\mathfrak{p}$ and $\mathfrak{q}$ given by \begin{gather*}
\mathfrak{p}\star_t\mathfrak{q} = \lim G[n] p\Theta_j[k]q G[m].
\end{gather*}

\begin{Proposition}The operation defined above is associative. Furthermore, the Burau representation is a functor between the categories of double cosets of ${\rm GL}(\infty)$ and of $B_\infty$.
\end{Proposition}
\begin{proof}The proof of the associative property is analogous to the proof of Theorem~\ref{th:asso}, using Lemma \ref{lem:thetas}. The functoriality follows from Lemma \ref{lem:thetas}.
\end{proof}

\section{Connections with other direct limits of groups}\label{sec:connexions}

We can extend the above constructions to the product $G^{[n]} = B_\infty\times\cdots \times B_\infty$ of $n$ copies of the infinite braid group. Let $K$ be the diagonal subgroup of $G^{[n]}$. Clearly, $K$ is naturally isomorphic to $B_\infty$. Let $K[\alpha]$ be the image of $B_\infty[\alpha]$ under this isomorphism. We define the product of double cosets componentwise.
 \begin{Corollary}Consider two double cosets \begin{gather*}
\mathfrak{p}\in K[\alpha]\backslash G^{[n]}\slash K[\beta], \qquad \mathfrak{q}\in K[\beta]\backslash G^{[n]}\slash K[\gamma],
\end{gather*} and let $p$ and $q$ be their respective representatives. Then the operation given by \begin{gather*}
\mathfrak{p}\circ \mathfrak{q} = K[\alpha] p\theta_j[\beta]q K[\beta],
\end{gather*} for $j$ sufficiently large, is well-defined and associative.
\end{Corollary}
\begin{proof}It follows from Propositions \ref{prop:suflarg}, \ref{prop:rep} and~\ref{th:asso}.
\end{proof}

When there is a surjective homomorphism $B_\infty$ onto a group $A$, we have an induced operation on the double cosets of $A$. More precisely, let $\psi\colon B_\infty \to A$ be a surjective homomorphism and consider, for each $\alpha\in \mathbb{N}$, the image $A[\alpha]$ of the subgroup $B_\infty[\alpha]$ by $\psi$. Then the induced product on the double cosets of $A$ with relation to the subgroups $A[\alpha]$ is well-defined. Indeed, this follows from the fact that the sequence used to define the product of double cosets in~$B_\infty$ not only converges, it becomes constant.

Let $S_\infty$ denote the infinite symmetric group, that is, the set of permutations $s\colon \mathbb{N}^\ast\to\mathbb{N}^\ast$ such that $s(i) = i$ for all but finitely many $i\in\mathbb{N}^\ast$ equipped with the composition of functions. Consider the subgroups $S_\infty[\alpha]\subset S_\infty$, $\alpha\in\mathbb{N}$ consisting of the elements fixing the set $\{1,2,\ldots,\alpha\}$ pointwise. In~\cite{neretin4}, Neretin defined an operation on the double cosets of $S_\infty$ with relation to the subgroups $S_\infty[\alpha]$ as follows: For integers $\beta\geq 0$ and $n>0$, denote by $\theta^s_n[\beta]$ the permutation given by \begin{gather*}
\theta^s_n[\beta](i) = \begin{cases} i+n, & \beta< i \leq \beta+n,\\
i-n,& \beta+n< i \leq \beta + 2n,\\
i, & \text{otherwise}.\end{cases}
\end{gather*} Given double cosets \begin{gather*}
\mathfrak{p}\in S_\infty[\alpha]\backslash S_\infty /S_\infty[\beta]\qquad\text{and}\qquad S_\infty[\beta]\backslash S_\infty /S_\infty[\gamma],
\end{gather*} and their representatives $p\in\mathfrak{p}$ and $q\in\mathfrak{q}$, define their product as \begin{gather*}
\mathfrak{p}\circ^{\!s}\mathfrak{q} = S_\infty[\alpha] p\theta^s_n[\beta]q S_\infty[\gamma],
\end{gather*} for sufficiently large~$n$.

On the other hand, consider the canonical homomorphism $j\colon B_\infty\to S_\infty$ that associates to each braid the corresponding permutation of its endpoints. It is clear that this is a surjective homomorphism (it is, up to conjugacy, the only surjective homomorphism from $B_\infty$ to $S_\infty$, see~\cite{artin2}) and hence induces an operation on the double cosets of~$S_\infty$ with relation to the subgroups $j(B_\infty[\alpha])$, $\alpha\in\mathbb{N}$. Furthermore, it is easy to check that $j(B_\infty[\alpha]) = S_\infty[\alpha]$ for each $\alpha\in\mathbb{N}$.

\begin{Proposition} The operation on double cosets of $S_\infty$, with relation to the subgroups $S_\infty[\alpha]$, $\alpha\in\mathbb{N}$, coincides with the operation induced by the group $B_\infty$.
\end{Proposition}

\begin{proof}In fact, it suffices to check that the elements $\theta^s_n[\beta]$ and $j(\theta_n[\beta])$ coincide for all integers $n>0$ and $\beta\geq 0$. But this identity follows directly from the definition of $\theta_n[\beta]$ and~$j$.
\end{proof}

As a last remark, we point out some similarities between the multiplicative structure defined in $B_\infty$ and that of $\operatorname{Aut}(F_\infty)$.
The group $\operatorname{Aut}(F_\infty)$ is defined as follows: Let $F_n$ be the free group with~$n$ generators $x_1,\ldots,x_n$ and denote by $Aut(F_n)$ the group of automorphisms of $F_n$. Then \begin{gather*}
\operatorname{Aut}(F_\infty) = \lim \operatorname{Aut}(F_n).
\end{gather*} The limit is taken with relation to the obvious inclusion $\operatorname{Aut}(F_n)\to\operatorname{Aut}(F_{n+1})$.

For each $\alpha\in\mathbb{N}$ consider the subgroup $H(\alpha)$ of $\operatorname{Aut}(F_\infty)$ of automorphisms $h$ such that $h(x_i) = x_i$ for $i\leq\alpha$. In~\cite{neretin6}, it is defined a product on the double cosets of $\operatorname{Aut}(F_\infty)$ in the following way: Consider the automorphism $\vartheta_j[\beta]\in\operatorname{Aut}(F_\infty)$ given by \begin{gather*}
\vartheta_j[\beta](x_i) = \begin{cases} x_i, &i\leq \beta, i>2j+\beta,\\
x_{i+j}, & \beta<i\leq \beta+j,\\
x_{i-j}, & \beta+j<i\leq \beta+2j.\end{cases}
\end{gather*} Then, for $p$ and $q$ in $\operatorname{Aut}(F_\infty)$, the product of the double cosets $H(\alpha) p H(\beta)$ and $H(\beta) q H(\gamma)$ is the double coset limit of the sequence $p\vartheta_j[m]q$ in $H(\alpha)\backslash \operatorname{Aut}(F_\infty)\slash H(\gamma)$.

 For each $n\in\mathbb{N}$, let $i_n\colon B_n\to \operatorname{Aut}(F_n)$ be the \emph{Artin representation} of $B_n$ on the free group~$F_n$, given by \begin{gather*}
i_n(\sigma_j)(x_k) = \begin{cases} x_j, & k=j+1,\\
x_jx_{j+1}x_j^{-1},& k=j,\\
x_k, & \text{otherwise}.\end{cases}
\end{gather*}This representation is faithful and therefore we can identify $B_n$ with the image of $i_n$ in $\operatorname{Aut}(F_n)$. Consider the limit homomorphism $i_\infty\colon B_\infty\to \operatorname{Aut}(F_\infty)$.
The element $\vartheta_j[m]$ is related to the image of the element $\theta_j[m]$ as we see in the following proposition
\begin{Proposition}Let $\beta$ be a fixed positive integer. For each $k\in \mathbb{N}$, consider the element $y_k = x_{\beta +k}x_{\beta + k - 1}\cdots x_{\beta+1}\in F_\infty$. Then \begin{gather*}
i_\infty(\theta_k[\beta])(x_i) = \begin{cases}
x_i,& i\leq \beta, i>2k+\beta,\\
y_k^{-1}x_{i+k}y_k,& \beta+1\leq i\leq k+\beta,\\
x_{i-k},& k+\beta< i\leq 2k+\beta.
\end{cases}
\end{gather*} In other words, \begin{gather*}
i_\infty(\theta_k[\beta])(x_i) = \begin{cases}
y_k^{-1}\vartheta_k[\beta](x_{i})y_k,& \beta+1\leq i\leq k+\beta,\\
\vartheta_k[\beta](x_{i}),& \text{otherwise}.
\end{cases}
\end{gather*}
\end{Proposition}

\begin{proof}For $k=1$ we have that $\theta_1[\beta] = \sigma_{\beta+1}$ and therefore
\begin{gather*}
i_\infty(\theta_1[\beta])(x_i) = i_\infty(\sigma_{\beta+1})(x_i) = \begin{cases}
x_i,& i\leq \beta, i>\beta+2,\\
x_i^{-1}x_{i+1}x_i,& i=1+\beta,\\
x_{i-1},& i= 2+\beta.
\end{cases}
\end{gather*} We are going to show the truth of the identity by induction on $k$. Suppose the identity holds for $k$. We can write $\theta_{k+1}[\beta]$ as \begin{gather*}
\theta_{k+1}[\beta] = \sigma_{k+\beta+1}\cdots\sigma_{2k+\beta+1}\theta_k[\beta]\sigma_{2k+\beta}\cdots\sigma_{k+\beta+1}.
\end{gather*}If we put $w = \sigma_{k+\beta+1}\cdots\sigma_{2k+\beta+1}$ and $s = \sigma_{2k+\beta}\cdots\sigma_{k+\beta+1}$ we can re-write the equation above as \begin{gather*}
\theta_{k+1}[\beta] = w\theta_k[\beta]s.
\end{gather*} We have five cases to analyze:

{\it Case 1.} When $\beta+1\leq i\leq k+\beta$, notice that $i_\infty(s)(x_i) = x_i$ and $i_\infty(\theta_k[\beta])(x_i) = y_k^{-1}x_{i+k}y_k$, therefore $i_\infty(\theta_{k+1}[\beta])(x_i) = i_\infty(w)\big(y_k^{-1}x_{i+k}y_k\big)$. Now,
\begin{gather*}
i_\infty(w)(x_{i+k}) = i_\infty (\sigma_{k+\beta+1}\cdots\sigma_{i+k-1})i_\infty(\sigma_{i+k})(x_{i+k}) \\
\hphantom{i_\infty(w)(x_{i+k})}{} =i_\infty(\sigma_{k+\beta+1}\cdots \sigma_{i+k-2})i_\infty(\sigma_{i+k-1})\big(x_{i+k}^{-1}x_{i+k+1}x_{i+k}\big) \\
\hphantom{i_\infty(w)(x_{i+k})}{} = i_\infty(\sigma_{k+\beta+1} \cdots\sigma_{i+k-3})i_\infty(\sigma_{i+k-2})\big(x_{i+k-1}^{-1}x_{i+k+1}x_{i+k-1}\big) = \cdots \\
\hphantom{i_\infty(w)(x_{i+k})}{} = i_\infty(\sigma_{k+\beta+1})\big(x_{k+\beta+2}^{-1}x_{k+i+1}x_{k+\beta+2}\big) = x_{k+\beta+1}^{-1}x_{k+i+1}x_{k+\beta+1}.
\end{gather*} Hence, \begin{gather*}
i_\infty(w)\big(y_k^{-1}x_{i+k}y_k\big) = y^{-1}_ki_\infty(w)(x_{i+k})y_k = y^{-1}_kx_{k+\beta+1}^{-1}x_{k+i+1}x_{k+\beta+1}y_k = y_{k+1}^{-1}x_{k+i+1}y_{k+1}.
\end{gather*}

{\it Case 2.} When $i=k+\beta+1$, we have that \begin{gather*}
i_\infty(s)(x_{k+i+\beta}) = x_{k+\beta+1}^{-1}\cdots x_{2k+\beta}^{-1}x_{2k+\beta+1}x_{2k+\beta}\cdots x_{k+\beta+1}.
\end{gather*} Hence, \begin{gather*}
i_\infty(\theta_k[\beta]s)(x_{k+\beta+1}) = i_\infty(\theta_k[\beta])\big(x_{k+\beta+1}^{-1}\cdots x_{2k+\beta}^{-1}x_{2k+\beta+1}x_{2k+\beta}\cdots x_{k+\beta+1}\big) \\
\hphantom{i_\infty(\theta_k[\beta]s)(x_{k+\beta+1})}{} =x_{\beta+1}^{-1}\cdots x_{k+\beta}^{-1}x_{2k+\beta+1}x_{k+\beta}\cdots x_{\beta+1} = y_k^{-1}x_{2k+\beta+1}y_k.
\end{gather*} Furthermore, $i_\infty(w)(x_{2k+\beta+1}) = x_{k+\beta+1}^{-1}x_{2k+\beta+2}x_{k+\beta+1}$ and hence
\begin{gather*}
i_\infty(\theta_{k+1}[\beta])(x_{k+\beta+1}) = i_\infty(w)\big(y_k^{-1}x_{2k+\beta+1}y_k\big) \\
\hphantom{i_\infty(\theta_{k+1}[\beta])(x_{k+\beta+1})}{} = y^{-1}_kx_{k+\beta+1}^{-1}x_{2k+\beta+2}x_{k+\beta+1}y_k = y_{k+1}^{-1}x_{2k+\beta+2}y_{k+1}.
\end{gather*}

{\it Case 3.} When $k+\beta+1< i\leq 2k+\beta+1$, it is sufficient to notice that $i_\infty(s)(x_i) = x_{i-1}$, $i_\infty(\theta_k[\beta])(x_{i-1}) = x_{i-k-1}$ and $i_\infty(w)(x_{i-k-1}) = x_{i-k-1}$.

{\it Case 4.} For the case $i=2k+\beta+2$ we have $i_\infty(\theta_k[\beta]s)(x_{2k+\beta+2}) = x_{2k+\beta+2}$. Furthermore, $i_\infty(w)(x_{2k+\beta+2}) = x_{k+\beta+1}$.

{\it Case 5.} For $i\leq \beta$ or $i>2k+\beta+2$, we have that $i_\infty(w)(x_i) = i_\infty(\theta_k[\beta](x_i) = i_\infty(s)(x_i) = x_i$ and the result follows.
\end{proof}

Thus the elements $\vartheta_j[\beta]$ and $i_\infty(\theta_j[\beta])$ are always conjugate in $\operatorname{Aut}(F_\infty)$ (in particular, by an element of $H(\beta)$). Nevertheless $i_\infty$ does not induce a homomorphism between the monoids of double cosets. In fact, consider the braid $\omega = \sigma_2^{-1}\sigma_3\sigma_1\sigma_3\sigma_2$ in $B_\infty$ and its projection $[\omega]$ in $B[2]\backslash B_\infty\slash B[2]$. Then $i_\infty(\omega\theta_N[2]\omega)$ and $i_\infty(\omega)\vartheta_N[2] i_\infty(\omega)$ do not belong to the same double coset of $H(2)\backslash \operatorname{Aut}(F_\infty) \slash H(2)$.

\subsection*{Acknowledgements}
This research was supported by FAPESP process 2015/03341-9. We are indebted to L.~Funar and A.K.M.~Libardi for useful discussions and continuous support.

\pdfbookmark[1]{References}{ref}
\LastPageEnding

\end{document}